\documentclass[12pt,leqno,a4paper]{article}

\usepackage[english]{babel}
\usepackage{amsmath}
\usepackage{amsfonts}
\usepackage{amsthm}
\usepackage{amssymb}

\usepackage{xspace}
\usepackage{euscript}
\usepackage{graphicx}
\usepackage{amscd}
\usepackage{tabularx}

\usepackage{enumerate}
 \usepackage{epsfig} 
 \usepackage{graphics} 
\theoremstyle{plain}
\newtheorem{theo}{Theorem}[section]
\newtheorem{prop}[theo]{Proposition}
\theoremstyle{definition}
\theoremstyle{remark}
\newtheorem{rem}[theo]{Remark}
\numberwithin{equation}{section}

\newcommand{\R}{\mathbb{R}}

\title{Quantitative estimates of strong unique continuation for anisotropic wave equations 
}

\author {Sergio Vessella\thanks{Universit\`a degli Studi di Firenze, Italy, E-mail:
\textsf{sergio.vessella@unifi.it}}}

\date{}

\begin{document}

\setcounter{section}{0}
\setcounter{secnumdepth}{2}

\maketitle

\begin{abstract}
The main results of the present paper consist in some quantitative estimates for solutions to the wave equation $\partial^2_{t}u-\mbox{div}\left(A(x)\nabla_x u\right)=0$.

Such estimates imply the following strong unique continuation properties:  (a) if $u$ is a solution to the the wave equation and $u$ is flat on a segment $\{x_0\}\times J$ on the $t$ axis, then $u$ vanishes in a neighborhood of $\{x_0\}\times J$. (b) Let u be a solution of the above wave equation in $\Omega\times J$ that vanishes on a a portion $Z\times J$ where $Z$ is a portion of $\partial\Omega$ and $u$ is flat on a segment $\{x_0\}\times J$, $x_0\in Z$, then $u$ vanishes in a neighborhood of $\{x_0\}\times J$. The property (a) has been proved by G. Lebeau, Comm. Part. Diff. Equat. 24 (1999), 777-783.
\medskip


\medskip

\end{abstract}
\section{Introduction} \label{introduction}

In this paper, we prove some quantitative estimates of strong unique
continuation for solutions to the wave equation
\begin{equation} \label{1-1-1}
\partial^2_{t}u-\mbox{div}\left(A(x)\nabla_x u\right)=0,
\end{equation}
($\mbox{div}:=\sum_{j=1}^n\partial_{x_j}$) where $A(x)$ is a real-valued symmetric $n\times n$, $n\geq 2$, matrix whose entries are functions of Lipschitz class and that satisfies the condition of uniform ellipticity. These estimates represent the quantitative counterparts of the following strong unique continuation property for equation \eqref{1-1-1}. Let $u$ be a weak solution to \eqref{1-1-1} and assume that

\begin{equation*}
\sup_{t\in J}\left\Vert u(\cdot,t) \right\Vert_{L^2\left(B_{r}\right)}=O(r^N)\mbox{, } \forall N\in\mathbb{N} \mbox{, as } r\rightarrow 0,
\end{equation*}
 where $J=(-T,T)$ is an interval of $\mathbb{R}$. Then we have
 \[u= 0 \quad \hbox{in } \mathcal{U},\]
 where $\mathcal{U}$ is a neighborhood of $\{0\}\times J$.
The above property of strong unique continuation was proved by Lebeau in \cite{Le}. Previously such a property was proved by Masuda \cite{Ma} whenever $J=\mathbb{R}$ and the entries of the matrix $A$ are functions of $C^2$ class and by Baouendi-Zachmanoglou \cite{Ba-Za} whenever the entries of $A$ are analytic functions. In both \cite{Ma} and \cite{Ba-Za}, the above property was proved also for first order perturbation of operator $\partial^2_{t}u-\mbox{div}\left(A(x)\nabla u\right)$. Also, we recall here the papers \cite{CheDY}, \cite{CheYZ} and \cite{Ra}. In such papers are proved unique continuation properties along and across lower dimensional manifolds for the wave equation.

Roughly speaking, the quantitative estimate of strong unique continuation (at the interior) that we prove is the following one (for the precise statement see Theorem \ref{5-115}). Let $u$ be a solution to \eqref{1-1-1} in the cylinder of $B_1\times J$, where, for any $R>0$, $B_R$ is the ball of $\mathbb{R}^n$, $n\geq 2$, centered at 0 with radius $R$ and let $r\in (0,1)$. Assume that

\begin{equation*}
\sup_{t\in J}\left\Vert u(\cdot,t) \right\Vert_{L^2\left(B_{r}\right)}\leq \varepsilon \quad\mbox{  and }\quad \left\Vert u(\cdot,0) \right\Vert_{H^2\left(B_{1}\right)}\leq 1,
\end{equation*}
where $\varepsilon<1$ then

\begin{gather}
\label{SUCP-introduction}
\left\Vert u(\cdot,0) \right\Vert_{L^2\left(B_{s_0}\right)} \leq C\left\vert\log \left(\varepsilon^{\theta}\right) \right\vert^{-1/6},
\end{gather}
where $s_0\in (0,1)$, $C\geq 1$ are constants independent by $u$ and $r$ and
\begin{equation}
\label{theta-introduction}
\theta= |\log r|^{-1}.
\end{equation}
For the well-known counterexample of John for the wave equation, \cite{J}, the logarithmic character of the estimate is not surprising,
but the novelty of estimate \eqref{SUCP-introduction} is the sharp dependence of the exponent $\theta$ on $r$. Indeed it is simple to check (see Remark \ref{SUCPrem}) that estimate \eqref{SUCP-introduction} implies the strong unique continuation property.

As a consequence of the above estimate and some reflection transformation introduced in \cite{AE}  we derive a quantitative estimate of unique continuation at the boundary (Theorem \ref{5-115Boundary}).

The proof of the above quantitative estimates of unique continuation is carried out exploiting the same ingredients used in \cite{Le} which are the Boman transformation, \cite{Bo}, and the application of Carleman estimate with singular weight, \cite{AKS}, \cite{hormanderpaper1}, \cite{EsVe} to the elliptic operator $\partial^2_{y}+\mbox{div}\left(A(x)\nabla_x\right)$. Also, see Section \ref{Concluding}, by using simple tricks we extend the quantitative estimate of strong unique continuation to the equations  $$q(x)\partial^2_{t}u-\mbox{div}\left(A(x)\nabla_x u\right)-b(x)\cdot\nabla_x u-c(x)u=0,$$ where $q$ is a positive function of Lipschitz class, $b$ is vector-valued function of Lipschitz class and $c$ is a bounded measurable function.

Now, it is worth while to remind that strong properties of unique continuation and the related quantitative estimates have been well understood for second order equation of elliptic (\cite{AE}, \cite{AKS}, \cite{hormanderpaper1}, \cite{KoTa1}) and parabolic type (\cite{A-V}, \cite{EsFe}, \cite{KoTa2}). The three sphere inequalities \cite{La}, doubling inequalities \cite{GaLi}, or two-sphere one cylinder inequality \cite{EsFeVe} are the typical form in which such quantitative estimates of unique continuation occur in the elliptic or in the parabolic context. For a more extensive literature on this subject, we refer to \cite{A-R-R-V} and \cite{Ve} for elliptic and parabolic equations respectively.

The main purpose that has led us to gain the estimates of the present paper is their applications in the stability issue for inverse hyperbolic problems with time independent unknown boundaries from transient data with a finite time of observation. For such problems some uniqueness results has been proved in \cite{Is2}. However, in contrast to the analogues problems for second order elliptic and parabolic equations, the stability issue in the hyperbolic context is much less studied. In a forthcoming paper, where we cover part of this lack, among the main tools that we use to prove sharp stability estimates there are precisely the quantitative estimate of unique continuation proved in the present paper. The quantitative estimate of strong unique continuation was applied for the first time to the elliptic inverse problems with unknown boundaries in \cite{A-B-R-V}. Concerning the parabolic inverse problems with unknown boundaries such estimates was applied in \cite{CRoVe1}, \cite{CRoVe2}, \cite{DcRVe}, \cite{Ve}. In both the cases, elliptic and parabolic, the stability estimates that was proved are optimal \cite{Dc-R} and \cite{Al} (elliptic case),  \cite{DcRVe} (parabolic case).

The plan of the paper is as follows. In Section \ref{MainRe} we state the main results of this paper, in Section \ref{SUCP_Estimates} we prove the theorems of Section \ref{MainRe}, in Section \ref{Concluding} we consider the case of the more general equation $q(x)\partial^2_{t}u-\mbox{div}\left(A(x)\nabla_x u\right)-b(x)\cdot\nabla_x u-c(x)u=0$.

\section{The main results}\label{MainRe}
\subsection{Notation and Definition} \label{sec:notation}
Let $n\in\mathbb{N}$, $n\geq 2$. For any $x\in \R^n$, we will denote $x=(x',x_n)$, where $x'=(x_1,\ldots,x_{n-1})\in\R^{n-1}$, $x_n\in\R$ and $|x|=\left(\sum_{j=1}^nx_j^2\right)^{1/2}$.
Given $x\in \R^n$, $r>0$, we will denote by $B_r$, $B'_r$ $\widetilde{B}_r$ the ball of $\mathbb{R}^{n}$, $\mathbb{R}^{n-1}$ and $\mathbb{R}^{n+1}$ of radius $r$ centered at 0 respectively.
For any open set $\Omega\subset\mathbb{R}^n$ and any function (smooth enough) $u$  we denote by $\nabla_x u=(\partial_{x_1}u,\cdots, \partial_{x_n})$ the gradient of $u$. Also, for the gradient of $u$ we use the notation $D_x$. If $j=0,1,2$ we denote by $D^j_x u$ the set of the derivatives of $u$ of order $j$, so $D^0_x u=u$, $D^1_x u=\nabla_x u$ and $D^2_x$ is the hessian matrix $\{\partial_{x_ix_j}u\}_{i,j=1}^n$. Similar notation are used whenever other variables occur and $\Omega$ is an open subset of $\mathbb{R}^{n-1}$ or a subset $\mathbb{R}^{n+1}$. By $H^{\ell}(\Omega)$, $\ell=0,1,2$ we denote the usual Sobolev spaces of order $\ell$, in particular we have $H^0(\Omega)=L^2(\Omega)$.

For any interval $J\subset \mathbb{R}$ and $\Omega$ as above we denote by
 \[\mathcal{W}\left(J;\Omega\right)=\left\{u\in C^0\left(J;H^2\left(\Omega\right)\right): \partial_t^\ell u\in C^0\left(J;H^{2-\ell}\left(\Omega\right)\right), \ell=1,2\right\}.\]

We shall use the letters $C,C_0,C_1,\cdots$ to denote constants. The value of the constants may change from line to line, but we shall specified their dependence everywhere they appear.

\subsection{Statements of the main results}\label{QEsucp}
Let $A(x)=\left\{a^{ij}(x)\right\}^n_{i,j=1}$ be a real-valued symmetric $n\times n$ matrix whose entries are measurable functions and they satisfy the following conditions for given constants $\rho_0>0$, $\lambda\in(0,1]$ and $\Lambda>0$,
\begin{subequations}
\label{1-65}
\begin{equation}
\label{1-65a}
\lambda\left\vert\xi\right\vert^2\leq A(x)\xi\cdot\xi\leq\lambda^{-1}\left\vert\xi\right\vert^2, \quad \hbox{for every } x, \xi\in\mathbb{R}^n,
\end{equation}
\begin{equation}
\label{2-65}
\left\vert A(x)-A(y)\right\vert\leq\frac{\Lambda}{\rho_0} \left\vert x-y \right\vert, \quad \hbox{for every } x, y\in\mathbb{R}^n.
\end{equation}
\end{subequations}

Let $q=q(x)$ be a a real-valued measurable function that satisfies
\begin{subequations}
\label{3-65}
\begin{equation} \label{3-65a}
\lambda\leq q(x)\leq\lambda^{-1}, \quad \hbox{for every } x\in\mathbb{R}^n,
\end{equation}
\begin{equation} \label{3'-65}
\left\vert q(x)-q(y)\right\vert\leq\frac{\Lambda}{\rho_0} \left\vert x-y \right\vert, \quad \hbox{for every } x, y\in\mathbb{R}^n.
\end{equation}
\end{subequations}

Let $u\in\mathcal{W}\left([-\lambda\rho_0,\lambda\rho_0];B_{\rho_0}\right)$ be a weak solution to

\begin{equation} \label{4i-65}
q(x)\partial^2_{t}u-\mbox{div}\left(A(x)\nabla_x u\right)=0, \quad \hbox{in } B_{\rho_0}\times(-\lambda\rho_0,\lambda\rho_0).
\end{equation}
Let $r_0\in (0,\rho_0]$ and denote by

\begin{equation} \label{4ii-65}
\varepsilon:=\sup_{t\in (-\lambda\rho_0,\lambda\rho_0)}\left(\rho_0^{-n}\int_{B_{r_0}}u^2(x,t)dx\right)^{1/2}
\end{equation}
and

\begin{equation} \label{4iii-65}
H:=\left(\sum_{j=0}^2\rho_0^{j-n}\int_{B_{\rho_0}}\left\vert D_x^ju(x,0)\right\vert^2 dx\right)^{1/2}.
\end{equation}

\begin{theo}[\textbf{estimate at the interior}]\label{5-115}
Let $u\in\mathcal{W}\left([-\lambda\rho_0,\lambda\rho_0];B_{\rho_0}\right)$ be a weak solution to \eqref{4i-65} and let \eqref{1-65} and \eqref{3-65} be satisfied. Then there exist constants $s_0\in (0,1)$ and $C\geq 1$ depending on $\lambda$ and $\Lambda$ only such that for every $0<r_0\leq \rho\leq s_0 \rho_0$ the following inequality holds true

\begin{gather}
\label{SUCP}
\left\Vert u(\cdot,0) \right\Vert_{L^2\left(B_{\rho}\right)} \leq \frac{C\left(\rho_0\rho^{-1}\right)^{C}(H+e\varepsilon)}{\left(\theta\log \left( \frac{H+e\varepsilon}{\varepsilon}\right) \right)^{1/6}},
\end{gather}
where
\begin{equation}
\label{theta}
\theta=\frac{\log (\rho_0/C\rho)}{\log (\rho_0/r_0)}.
\end{equation}
\end{theo}
The proof of Theorem \ref{5-115} is given in subsection \ref{SUCP interior}.

\bigskip

\begin{rem}\label{SUCPrem}
Observe that estimate \eqref{SUCP} implies the following property of strong unique continuation. Let $u\in\mathcal{W}\left([-\lambda\rho_0,\lambda\rho_0];B_{\rho_0}\right)$ be a weak solution to \eqref{4i-65} and assume that
$$\sup_{t\in (-\lambda\rho_0,\lambda\rho_0)}\left(\rho_0^{-n}\int_{B_{r_0}}u^2(x,t)dx\right)^{1/2}=O(r_0^N)\mbox{, } \forall N\in\mathbb{N} \mbox{, as } r_0\rightarrow 0,$$
then
\begin{equation}
\label{ucp-rem}
u(\cdot,t)=0 \mbox{, for } |x|+\lambda^{-1}s_0|t|\leq s_0 \rho_0.
\end{equation}

It is enough to consider the case $t=0$. We argue by contradiction and assume that
\begin{equation}
\label{ucp-rem-1}
\left\Vert u(\cdot,0) \right\Vert_{L^2\left(B_{s_0\rho_0}\right)}>0.
\end{equation}
Hence it is not restrictive to assume that $H=\left\Vert u(\cdot,0) \right\Vert_{H^2\left(B_{\rho_0}\right)}=1$. Now we apply inequality \eqref{SUCP} with $\varepsilon_0=C_Nr_0^N$, $N\in\mathbb{N}$, and passing to the limit as $r_0\rightarrow 0$ we have that \eqref{SUCP} implies
\begin{equation*}
\label{ucp-rem-2}
\left\Vert u(\cdot,0) \right\Vert_{L^2\left(B_{s_0\rho_0}\right)}\leq Cs_0^{-C}N^{-1/6} \mbox{, } \forall N\in\mathbb{N},
\end{equation*}
by passing again to the limit as $N\rightarrow 0$ we get $\left\Vert u(\cdot,0) \right\Vert_{L^2\left(B_{\rho}\right)}=0$ that contradicts \eqref{ucp-rem-1}. By \eqref{ucp-rem} and by UCP property proved by \cite{hormanderpaper2}, \cite{Ro-Zu},  \cite{Ta},  see also \cite{isakovlib2} we have that,  if the entries of $A$ are function in $C^{\infty}(\mathbb{R}^n)$ then $u=0$ in the domain of influence of $\{0\}\times (-\lambda\rho_0,\lambda\rho_0)$.
\end{rem}

\bigskip

In order to state Theorem \ref{5-115Boundary} below let us introduce some notation.
Let $\phi$ be a function belonging to $C^{1,1}\left(B^{\prime}_{\rho_0}\right)$ that satisfies

\begin{equation}
\label{phi_0}
\phi(0)=\left\vert\nabla_{x'}\phi(0)\right\vert=0
\end{equation}
and
\begin{equation}
\label{phi_M0}
\left\Vert\phi\right\Vert_{C^{1,1}\left(B^{\prime}_{\rho_0}\right)}\leq E\rho_0,
\end{equation}
where

\begin{equation*}
\left\Vert\phi\right\Vert_{C^{1,1}\left(B^{\prime}_{\rho_0}\right)}=\left\Vert\phi\right\Vert_{L^{\infty}\left(B^{\prime}_{\rho_0}\right)}
+\rho_0\left\Vert\nabla_{x'}\phi\right\Vert_{L^{\infty}\left(B^{\prime}_{\rho_0}\right)}+
\rho_0^2\left\Vert D_{x'}^2\phi\right\Vert_{L^{\infty}\left(B^{\prime}_{\rho_0}\right)}.
\end{equation*}
For any $r\in (0,\rho_0]$ denote by
\[
K_{r}:=\{(x',x_n)\in B_{r}: x_n>\phi(x')\}
\]
and
\[Z:=\{(x',\phi(x')): x' \in B^{\prime}_{\rho_0}\}.\]

Let $u\in \mathcal{W}\left([-\lambda\rho_0,\lambda\rho_0];K_{\rho_0}\right)$ be a solution to

\begin{equation} \label{4i-65Boundary}
\partial^2_{t}u-\mbox{div}\left(A(x)\nabla_x u\right)=0, \quad \hbox{in } K_{\rho_0}\times(-\lambda\rho_0,\lambda\rho_0),
\end{equation}
satisfying one of the following conditions
\begin{equation} \label{DirichBoundary}
u=0, \quad \hbox{on } Z\times(-\lambda\rho_0,\lambda\rho_0)
\end{equation}
or
\begin{equation} \label{NeumBoundary}
A\nabla_x u\cdot \nu=0, \quad \hbox{on } Z\times(-\lambda\rho_0,\lambda\rho_0),
\end{equation}
where $\nu$ denotes the outer unit normal to $Z$.

Let $r_0\in (0,\rho_0]$ and denote by

\begin{equation} \label{4ii-65Boundary}
\varepsilon=\sup_{t\in (-\lambda\rho_0,\lambda\rho_0)}\left(\rho_0^{-n}\int_{K_{r_0}}u^2(x,t)dx\right)^{1/2}
\end{equation}
and

\begin{equation} \label{4iii-65Boundary}
H=\left(\sum_{j=0}^2\rho_0^{j-n}\int_{K_{\rho_0}}\left\vert D_x^ju(x,0)\right\vert^2 dx\right)^{1/2}.
\end{equation}

\begin{theo}[\textbf{estimate at the boundary}]\label{5-115Boundary}
Let \eqref{1-65} be satisfied. Let $u\in\mathcal{W}\left([-\lambda\rho_0,\lambda\rho_0];K_{\rho_0}\right)$ be a solution to \eqref{4i-65Boundary} satisfying \eqref{4ii-65Boundary} and \eqref{4iii-65Boundary}. Assume that $u$ satisfies either \eqref{DirichBoundary} or \eqref{NeumBoundary}. There exist constants $\overline{s}_0\in (0,1)$ and $C\geq 1$ depending on $\lambda$, $\Lambda$ and $E$ only such that for every $0<r_0\leq \rho\leq \overline{s}_0 \rho_0$ the following inequality holds true

\begin{gather}
\label{SUCPBoundary}
\left\Vert u(\cdot,0) \right\Vert_{L^2\left(K_{\rho}\right)}\leq \frac{C\left(\rho_0\rho^{-1}\right)^{C}(H+e\varepsilon)}{\left(\widetilde{\theta}\log \left( \frac{H+e\varepsilon}{\varepsilon}\right)\right)^{1/6}} ,
\end{gather}
where
\begin{equation}
\label{theta}
\widetilde{\theta}=\frac{\log (\rho_0/C\rho)}{\log (\rho_0/r_0)}.
\end{equation}
\end{theo}
The proof of Theorem \ref{5-115Boundary} is given in subsection \ref{SUCP boundary}.

\bigskip

\begin{rem}\label{SUCPremBound}
By arguing similarly to Remark \ref{SUCPrem} we have that estimate \eqref{SUCPBoundary} implies the following property of strong unique continuation at the boundary. Let $u\in\mathcal{W}\left([-\lambda\rho_0,\lambda\rho_0];K_{\rho_0}\right)$ be a solution to \eqref{4i-65Boundary} satisfying either \eqref{DirichBoundary} or \eqref{NeumBoundary} and assume that
$$\sup_{t\in (-\lambda\rho_0,\lambda\rho_0)}\left(\rho_0^{-n}\int_{K_{r_0}}u^2(x,t)dx\right)^{1/2}=O(r_0^N)\mbox{, } \forall N\in\mathbb{N} \mbox{, as } r_0\rightarrow 0,$$
then
\begin{equation*}
u(x,t)=0 \mbox{, for } x\in K_{\rho(t)} \mbox{, } t\in (-\lambda\rho_0,\lambda\rho_0),
\end{equation*}
where $\rho(t)=\overline{s}_0\left(\rho_0-\lambda^{-1}|t|\right)$
\end{rem}

\section{Proof of Theorems \ref{5-115}, \ref{5-115Boundary}} \label{SUCP_Estimates}
\subsection{Proof of Theorem \ref{5-115}}\label{SUCP interior}
Observe that to prove Theorem \ref{5-115} we can assume that $u(x,t)$ is even with respect to the variable $t$. Indeed defining
\[u_+(x,t)=\frac{u(x,t)+u(x,-t)}{2},\]
we see that $u_+$ satisfies all the hypotheses of Theorem \ref{5-115} and, in particular, we have
\[u_+(x,0)=u(x,0),\]
and
\[\varepsilon=\sup_{t\in (-\lambda\rho_0,\lambda\rho_0)}\left(\rho_0^{-n}\int_{B_{r_0}}u_+^2(x,t)dx\right)^{1/2}\quad \hbox{, } H=\left(\sum_{j=0}^2\rho_0^{j-n}\int_{B_{\rho_0}}\left\vert D_x^ju_+(x,0)\right\vert^2 dx\right)^{1/2}.\]
Hence, from now on we assume that $u(x,t)$ is even with respect to the variable $t$. Moreover it is not restrictive to assume $\rho_0=1$.

\bigskip

In order to prove Theorem \ref{5-115} we prove some preliminary propositions.

Let us start by introducing an extension $\widetilde{u}_0$ of the function $u_0:=u(\cdot,0)$ such that $\widetilde{u}_0\in H^2\left(B_2\right)$ and

\begin{equation} \label{2-70}
\|\widetilde{u}_0\|_{H^2\left(B_2\right)}\leq CH,
\end{equation}
where $C$ is an absolute constant.

Let us denote by $\lambda_j$, with $0<\lambda_1\leq\lambda_2\leq \cdots\leq\lambda_j\leq\cdots$ the eigenvalues associated to the Dirichlet problem

\begin{equation}
\label{1-71}
\left\{\begin{array}{ll}
\mbox{div}\left(A(x)\nabla_x v\right)+\omega q(x)v=0, & \textrm{in }B_2,\\[2mm]
v\in H^1\left(B_2\right),
\end{array}\right.
\end{equation}
and by $e_j(\cdot)$ the corresponding eigenfunctions normalized by

\begin{equation} \label{2-71}
\int_{B_2}e^2_j(x)q(x)dx=1.
\end{equation}
Since, by Poincar\'{e} inequality we have

\begin{gather*}
\lambda_j=\int_{B_2}\lambda_j e^2_j(x)q(x)dx=\int_{B_2} A(x)\nabla_x e_j(x)\cdot \nabla_x e_j(x) dx\geq \\
\geq\lambda \int_{B_2}\left\vert\nabla_x e_j(x)\right\vert^2 dx\geq c\lambda \int_{B_2}e^2_j(x)dx\geq c\lambda^2 \int_{B_2}e^2_j(x)q(x)dx=c\lambda^2,
\end{gather*}
where $c$ is an absolute constant, we get

\begin{equation} \label{2-71NEW}
\lambda_j\geq c\lambda^2 \mbox{, for every } j\in\mathbb{N}.
\end{equation}

Denote by

\begin{equation} \label{4-71}
\alpha_j :=\int_{B_2}\widetilde{u}_0(x) e_j(x)q(x)dx
\end{equation}
and let

\begin{equation} \label{3-71}
\widetilde{u}(x,t):=\sum_{j=1}^{\infty}\alpha_j e_j(x)\cos\sqrt{\lambda_j} t.
\end{equation}

\begin{prop}\label{pag71}
We have
\begin{equation} \label{1-72}
\sum_{j=1}^{\infty}\left(1+\lambda_j^2\right)\alpha^2_j\leq C H^2,
\end{equation}
where $C$ depends on $\lambda, \Lambda$ only.
Moreover, $\widetilde{u}\in \mathcal{W}\left(\mathbb{R};B_2\right)\cap C^0\left(\mathbb{R};H^2\left(B_2\right)\cap H^1_0\left(B_2\right)\right)$ is an even function with respect to variable $t$ and it satisfies
\begin{equation}
\label{2-72}
\left\{\begin{array}{ll}
q(x)\partial^2_{t}\widetilde{u}-\mbox{div}\left(A(x)\nabla_x \widetilde{u}\right)=0, \quad \hbox{in } B_2\times \mathbb{R},\\[2mm]
\widetilde{u}(\cdot,0)=\widetilde{u}_0, \quad \hbox{in } B_2,\\[2mm]
\partial_t\widetilde{u}(\cdot,0)=0, \quad \hbox{in } B_2.
\end{array}\right.
\end{equation}

\end{prop}

\begin{proof}
By \eqref{1-71} and \eqref{2-71} we have
\begin{gather*}
\lambda_j\alpha_j=\int_{B_2}\widetilde{u}_0(x) \lambda_jq(x)e_j(x)dx=\\
=-\int_{B_2}\mbox{div}\left(A(x)\nabla_x \widetilde{u}_0(x)\right)e_j(x)dx.
\end{gather*}
Hence, by \eqref{1-65}, \eqref{3-65} and \eqref{2-70} we have
\begin{equation*}
\sum_{j=1}^{\infty}\left(1+\lambda_j^2\right)\alpha^2_j=\|\widetilde{u}_0\|^2_{L^2\left(B_2;qdx\right)}+\left \Vert {\frac{1}{q}\mbox{div}\left(A\nabla_x \widetilde{u}_0\right)}\right \Vert^2_{L^2\left(B_2;qdx\right)}\leq C H^2,
\end{equation*}
where $C$ depends on $\lambda, \Lambda$ only and \eqref{1-72} follows.
\end{proof}

Note that, for the uniqueness to the Cauchy problem for equation \eqref{4i-65} \cite{Ev}, we have
\begin{equation}\label{unic}
\widetilde{u}(x,t)=u_+(x,t), \quad \mbox{ for } |x|+\lambda^{-1}|t|< 1.
\end{equation}

\bigskip

Let us introduce the following nonnegative, even function  $\psi$ such that
\begin{equation}\label{modcont}
\psi(t)=\left\{\begin{array}{ccc}
                   \frac{1}{2}\left(1+\cos\pi t\right), &\mbox{for}& |t|\leq 1, \\
                   0, & \mbox{for}&|t|>1.
                 \end{array}
\right.
\end{equation}
Notice that $\psi\in C^{1,1}$, $\mbox{supp }\psi=[0,1]$ and
\begin{equation}
\label{1iii-73}
\int_{\mathbb{R}}\psi(t)dt=1.
\end{equation}

Let
\begin{equation}
\label{2-74}
\widehat{\psi}(\tau)=\int_{\mathbb{R}}\psi(t)e^{-i\tau t}dt=\int_{\mathbb{R}}\psi(t) \cos \tau t dt \mbox{, } \tau\in\mathbb{R}.
\end{equation}
Since $\psi$ has compact support, $\widehat{\psi}$ is an entire function.
By \eqref{1iii-73} we have
\begin{equation*}
\left|\widehat{\psi}(\tau)\right|\leq\int_{\mathbb{R}}\psi(t)dt=1 \mbox{, for every } \tau\in\mathbb{R},
\end{equation*}
and
\begin{gather*}
\left|\tau^2\widehat{\psi}(\tau)\right|=\left|-\int_{\mathbb{R}}\psi(t)\frac{d^2}{dt^2} \cos \tau tdt\right|=\left|-\int_{\mathbb{R}}\psi^{''}(t) \cos \tau tdt\right|\leq \pi^2 \mbox{, for every } \tau\in\mathbb{R},
\end{gather*}
hence we have
\begin{equation}
\label{3-74}
\left|\widehat{\psi}(\tau)\right|\leq\min\left\{1,\pi^2\tau^{-2}\right\} \mbox{, for every } \tau\in\mathbb{R}.
\end{equation}

Let

\begin{equation}
\label{1-75}
\vartheta(t)=4\lambda^{-1}\psi(4\lambda^{-1}t) \mbox{, } t\in\mathbb{R}.
\end{equation}

\bigskip

In the following proposition we collect the elementary properties of $\vartheta$ that we need.
\begin{prop}\label{proptheta}
The function $\vartheta$ is an even and positive function such that $\vartheta\in C^{1,1}$, $\mbox{supp }\vartheta=\left[-\frac{\lambda}{4},\frac{\lambda}{4}\right]$, $\int_{\mathbb{R}}\vartheta(t)dt=1$, $\widehat{\vartheta}(\tau)=\widehat{\psi}\left(\frac{\tau}{4}\right)$ and

\begin{equation}
\label{2'-75}
\int_{\mathbb{R}}\left|\vartheta^{\prime}(t)\right|dt=8\lambda^{-1},
\end{equation}

\begin{equation}
\label{3'-75}
\left|\widehat{\vartheta}(\tau)\right|\leq\min\left\{1,16\pi^2(\tau\lambda)^{-2}\right\} \mbox{, for every } \tau\in\mathbb{R},
\end{equation}
\begin{equation} \label{1-110}
\left\vert \widehat{\vartheta}(\tau)-1\right\vert\leq\left(\frac{\lambda\tau}{4}\right)^2, \quad \hbox{for } \left\vert\frac{\lambda\tau}{4} \right\vert\leq\frac{\pi}{2},
\end{equation}
\begin{equation} \label{1'-110}
\frac{1}{2}\leq \widehat{\vartheta}(\tau), \quad \hbox{for } \left\vert\frac{\lambda\tau}{4} \right\vert\leq\frac{1}{\sqrt{2}}\mbox{ .  }
\end{equation}

\end{prop}
\begin{proof}
We limit ourselves to prove property \eqref{1-110} and \eqref{1'-110}, since the other properties are immediate consequences  of \eqref{2-74}, \eqref{3-74} and \eqref{1-75}.
We have

\begin{equation}
\label{n1}
\left\vert \widehat{\vartheta}(\tau)-1\right\vert\leq \int^1_{-1} \psi(s)\left(1-\cos\left(\frac{\lambda s \tau}{4}\right)\right)ds.
\end{equation}
Now, if $s\in [-1,1]$ and $\left\vert\frac{\lambda\tau}{4} \right\vert\leq\frac{\pi}{2}$ then
\begin{equation*}
1-\cos\left(\frac{\lambda s \tau}{4}\right)\leq \left(\frac{\lambda\tau}{4}\right)^2 .
\end{equation*}
Hence by \eqref{n1} we get \eqref{1-110}. Finally \eqref{1'-110} is an immediate consequence of \eqref{1-110}
\end{proof}

As usual, if $f,g\in L^1(\mathbb{R})$, we denote by $(f\ast g)(t):=\int_{\mathbb{R}}f(t-s)g(s)ds$. Moreover we denote by $f^{\ast (k)}:=f\ast f^{\ast (k-1)}$, for $k\geq 2$, where $f^{\ast (1)}:=f$.

Let us define
\begin{equation}
\label{2-75}
\vartheta_k(t):=\left(k\vartheta(kt)\right)^{\ast (k)} \mbox{, for every } k\in\mathbb{N}.
\end{equation}
Notice that $\vartheta_k\geq 0$, $\mbox{supp }\vartheta_k\subset\left[-\frac{\lambda}{4},\frac{\lambda}{4}\right]$, $\int_{\mathbb{R}}\vartheta_k(t)dt=1$, for every $k\in\mathbb{N}$ and
\begin{equation}
\label{fourierthetak}
\widehat{\vartheta}_k(\tau)=\left(\widehat{\vartheta}(k^{-1} \tau)\right)^k \mbox{, for every } k\in\mathbb{N}, \tau\in\mathbb{R}.
\end{equation}
Moreover, by \eqref{1'-110} we have

\begin{equation}
\label{limit}
\lim \limits_{k\rightarrow +\infty} \widehat{\vartheta}_k(\tau)=1 \mbox{, for every } \tau\in\mathbb{R}.
\end{equation}

For any number $\mu\in(0,1]$ and any $k\in\mathbb{N}$ let us set

\begin{equation}
\label{4-75}
\varphi_{\mu,k}=\left(\vartheta_k\ast \varphi_{\mu}\right),
\end{equation}
where

\begin{equation}
\label{4'-75}
\varphi_{\mu}(t)=\mu^{-1}\vartheta\left(\mu^{-1}t\right), \mbox{ for every } t\in\mathbb{R}.
\end{equation}

We have $\mbox{supp }\varphi_{\mu,k}\subset\left[-\frac{\lambda(\mu+1)}{4},\frac{\lambda(\mu+1)}{4}\right]$, $\varphi_{\mu,k}\geq 0$ and $\int_{\mathbb{R}}\varphi_{\mu,k}(t)dt=1$.

Now, let us define the following slight different form of the Boman transformation of $\widetilde{u}(x,\cdot)$, \cite{Bo},
\begin{equation}
\label{2-76}
\widetilde{u}_{\mu,k}(x)=\int_{\mathbb{R}}\widetilde{u}(x,t)\varphi_{\mu,k}(t)dt \mbox{, for } x\in B_2.
\end{equation}
\begin{prop}\label{4-97}
If $k\in\mathbb{N}$ and $\mu=k^{-1/6}$ then the following inequality holds true
\begin{equation}
\label{2'-97}
\left\Vert u(\cdot,0)-\widetilde{u}_{\mu,k} \right\Vert_{L^2 \left(B_{1}\right)}\leq C H k^{-1/6},
\end{equation}
where C depends on $\lambda$ only.
\end{prop}

\begin{proof}

Let $\mu\in (0,1]$. By applying the triangle inequality and taking into account \eqref{1iii-73} and \eqref{4'-75} we have
\begin{gather}\label{1-100}
\left \Vert u(\cdot,0)-\widetilde{u}_{\mu ,k}(\cdot)\right \Vert_{L^2 \left(B_1\right)}\leq \\ \nonumber
\leq\left( \int_{B_1}dx\int_{-\lambda \mu /4}^{\lambda \mu /4}\left \vert u(x,0)-%
\widetilde{u}(x,t)\right \vert ^{2}\varphi _{\mu }(t)dt\right) ^{1/2}+\\ \nonumber
+\left(\int_{B_1}dx
\int_{-\lambda \left( \mu +1\right) /4}^{\lambda \left( \mu +1\right)
/4}\left \vert \widetilde{u}(x,t)\right \vert ^{2}dt\right) ^{1/2}\left \Vert
\varphi _{\mu }-\varphi _{\mu ,k}\right \Vert _{L^{2}\left( \mathbb{R}\right)
}:=I_1+I_2.
\end{gather}
In order to estimate from above $I_1$ we observe that by the energy inequality, \eqref{2-70} and by taking into account that $\partial_t \widetilde{u}(x,0)=0$, we have

\begin{eqnarray*}
 \int_{B_2}\left \vert \partial_t \widetilde{u}(x,t)
\right \vert ^{2}dx\leq\int_{B_2}\left(\left \vert \partial_t \widetilde{u}(x,t)
\right \vert ^{2}+\left \vert \nabla_x \widetilde{u}(x,t)
\right \vert ^{2}\right)dx \leq \\
\leq \lambda^{-2}\int_{B_2}\left(\left \vert \partial_t \widetilde{u}(x,0)
\right \vert ^{2}+\left \vert \nabla_x \widetilde{u}(x,0)
\right \vert ^{2}\right)dx\leq CH^2,
\end{eqnarray*}
where $C$ depends on $\lambda$ only.
Therefore

\begin{gather*}
I^2_1\leq 2\int_{B_{1}}dx\left \vert \int_{0}^{\lambda \mu /4}\partial _{\eta }
\widetilde{u}(x,\eta )d\eta \right \vert ^{2}\leq \frac{\lambda \mu }{2}
\int_{B_{1}}dx\int_{0}^{\lambda \mu /4}\left \vert \partial _{\eta }
\widetilde{u}(x,\eta )\right \vert ^{2}d\eta \leq CH^2\mu ^{2}.
\end{gather*}
Hence
\begin{equation}\label{1-101}
I_1\leq CH\mu,
\end{equation}
where $C$ depends on $\lambda$ only.

Concerning $I_2$, first we observe that by using Poincar\'{e} inequality, energy inequality and \eqref{2-70} we have
\begin{gather}\label{1-107}
\int_{-\lambda \left( \mu +1\right) /4}^{\lambda \left( \mu +1\right)
/4}dt\int_{B_1}\left \vert \widetilde{u}(x,t)\right \vert ^{2}dx\leq\int_{-\lambda/2 }^{\lambda/2}dt\int_{B_2}\left \vert \widetilde{u}(x,t)\right \vert ^{2}dx\leq\\ \nonumber
\leq C\int_{-\lambda/2 }^{\lambda/2}dt\int_{B_2}\left \vert \nabla_x\widetilde{u}(x,t)\right \vert ^{2}dx\leq CH^2,
\end{gather}
where $C$ depends on $\lambda$ only.

In order to estimate from above $\left \Vert\varphi _{\mu }-\varphi _{\mu ,k}\right \Vert _{L^{2}\left( \mathbb{R}\right)}$
we recall that $\widehat{\varphi}_\mu(\tau)=\widehat{\vartheta}(\mu\tau)$ and $\widehat{\varphi}_{\mu,k}(\tau)=\widehat{\vartheta}(\mu\tau)\left(\widehat{\vartheta}(k^{-1}\tau)\right)^k$, hence the Parseval identity and a change of variable give

\begin{equation}\label{1-102}
2\pi\left \Vert\varphi _{\mu }-\varphi _{\mu ,k}\right \Vert _{L^{2}\left( \mathbb{R}\right)}^2=\frac{1}{\mu}\int_{\mathbb{R}}\left\vert \left(\widehat{\vartheta}((\mu k)^{-1}\tau)\right)^k-1\right\vert^2 \left\vert\widehat{\vartheta}(\tau)\right\vert^2 d\tau.
\end{equation}
By \eqref{3'-75}, \eqref{1-110} and \eqref{1'-110} and by using the elementary inequalities $1-e^{-z}\leq z$, for every $z\in \mathbb{R}$, and $\log s\leq s-1$, for every $s>0$,  we have, whenever $\left\vert\frac{\lambda\tau}{4\mu k}\right\vert\leq\frac{1}{\sqrt{2}}$,

\begin{equation}\label{4-104}
0\leq 1-\left(\widehat{\vartheta}((\mu k)^{-1}\tau)\right)^k=1-e^{k\log \widehat{\vartheta}((\mu k)^{-1}\tau)}\leq\frac{\lambda^2\tau^2}{8\mu^2k} .
\end{equation}

Now let $\delta\in(0,1]$ be a number that we shall choose later and denote by $\beta=\frac{4\mu k}{\sqrt{2}\lambda}\delta$. By \eqref{1-102}, \eqref{3'-75} and \eqref{4-104} we have

\begin{gather}\label{2-105}
2\pi\left \Vert\varphi _{\mu }-\varphi _{\mu ,k}\right \Vert _{L^{2}\left( \mathbb{R}\right)}^2=\frac{1}{\mu}\int_{|\tau|\leq\beta}\left\vert \left(\widehat{\vartheta}((\mu k)^{-1}\tau)\right)^k-1\right\vert^2 \left\vert\widehat{\vartheta}(\tau)\right\vert^2 d\tau+
\\ \nonumber
+\frac{1}{\mu}\int_{|\tau|\geq\beta}\left\vert \left(\widehat{\vartheta}((\mu k)^{-1}\tau)\right)^k-1\right\vert^2 \left\vert\widehat{\vartheta}(\tau)\right\vert^2 d\tau\leq \\ \nonumber
\leq\frac{1}{\mu}\int_{|\tau|\leq\beta}\left(\frac{\lambda^2\tau^2}{8\mu^2 k}\right)^2 d\tau+\frac{1}{\mu}\int_{|\tau|>\beta}\left(\frac{32\pi^2}{\lambda^2\tau^2}\right)^2 d\tau\leq C\left(k^3\delta^5+\frac{1}{\delta^3\mu^4 k^3}\right),
\end{gather}
where $C$ depends on $\lambda$ only.
If $\mu^2 k^{3/5}\geq 1$, we choose $\delta=(\mu^2 k^3)^{-1/4}$ and by \eqref{2-105} we have

\begin{equation}\label{1-114}
\left \Vert\varphi _{\mu }-\varphi _{\mu ,k}\right \Vert _{L^{2}\left( \mathbb{R}\right)}\leq C \left(k^{3/5}\mu^{2}\right)^{-5/8},
\end{equation}
where $C$ depends on $\lambda$ only.
Hence recalling \eqref{1-107} we have
\begin{gather}\label{I2}
I_2 \leq C H\left(k^{3/5}\mu^{2}\right)^{-5/8}.
\end{gather}

By \eqref{1-100}, \eqref{1-101} and \eqref{1-101} we obtain

\begin{equation}
\label{1-114}
\left\Vert u(\cdot,0)-\widetilde{u}_{\mu,k} \right\Vert_{L^2 \left(B_{1}\right)}\leq C H\left(\mu+\left(k^{3/5}\mu^{2}\right)^{-5/8}\right) .
\end{equation}

Now, if $\mu=k^{-\frac{1}{6}}$, $k\geq 1$ then \eqref{1-114} implies \eqref{2'-97}.
\end{proof}

From now on we fix $\overline{\mu}:=k^{-\frac{1}{6}}$ for $k\geq 1$ and we set

\begin{equation}
\label{u-k}
\widetilde{u}_k:=\widetilde{u}_{\overline{\mu},k}.
\end{equation}

Let us introduce now, for every $k\in \mathbb{N}$ an even function $g_k\in C^{1,1}(\mathbb{R})$ such that if $|z|\leq k$ then we have $g_k(z)=\cosh z$, if  $|z|\geq 2k$ then we have $g_k(z)=\cosh 2k$ and such that it satisfies the condition

\begin{equation}
\label{2-80}
\left\vert g_k(z) \right\vert+\left\vert g^{\prime}_k(z) \right\vert+\left\vert g^{\prime\prime}_k(z) \right\vert\leq ce^{2k} \mbox{, for every } z\in\mathbb{R},
\end{equation}
where $c$ is an absolute constant.

The following proposition holds true.

\begin{prop}\label{2-81prop}
Let
\begin{equation}
\label{3-80}
v_{k}(x,y):=\sum_{j=1}^{\infty}\alpha_j \widehat{\varphi}_{\overline{\mu},k}\left(\sqrt{\lambda_j}\right)g_k\left(y\sqrt{\lambda_j}\right)e_j(x) \mbox{ , for  } (x,y)\in B_2\times\mathbb{R}.
\end{equation}
We have that $v_{k}(\cdot,y)$ belongs to $H^1\left(B_2\right)$ for every $y\in \mathbb{R}$, $v_{k}(x,y)$ is an even function with respect to $y$ and it satisfies
\begin{equation}
\label{4-5-6-81}
\left\{\begin{array}{ll}
q(x)\partial^2_{y}v_{k}+\mbox{div}\left(A(x)\nabla_x v_{k}\right)=f_{k}(x,y), \quad \hbox{in } B_2\times \mathbb{R},\\[2mm]
v_{k}(\cdot,0)=\widetilde{u}_{k},\quad \hbox{in } B_2.
\end{array}\right.
\end{equation}
where
\begin{equation}
\label{1-81}
f_{k}(x,y)=\sum_{j=1}^{\infty}\lambda_j\alpha_j \widehat{\varphi}_{\overline{\mu},k}\left(\sqrt{\lambda_j}\right)\left(g^{\prime\prime}_k\left(y\sqrt{\lambda_j}\right)-
g_k\left(y\sqrt{\lambda_j}\right)\right)e_j(x).
\end{equation}
Moreover we have

\begin{equation}
\label{3-81}
\sum_{j=0}^{2}\|\partial^{j}_yv_{k}(\cdot,y)\|_{H^{2-j}\left(B_2\right)}\leq CH e^{2k} \mbox{, for every  } y\in \mathbb{R},
\end{equation}

\begin{equation}
\label{2-81}
\|f_{k}(\cdot,y)\|_{L^2\left(B_2\right)}\leq CH e^{2k}\min\left\{1,\left(4\pi\lambda^{-1}|y|\right)^{2k}\right\} \mbox{, for every  } y\in \mathbb{R},
\end{equation}
where $C$ depends on $\lambda$ and $\Lambda$ only, and

\begin{equation}
\label{1-82}
\|v_{k}(\cdot,0)\|_{L^2\left(B_{r_0}\right)}\leq \varepsilon.
\end{equation}
\end{prop}
\begin{proof}
First of all observe that
\begin{equation}
\label{1-83}
\left|\widehat{\varphi}_{\overline{\mu},k}\left(\sqrt{\lambda_j}\right)\right|\leq \|\varphi_{\overline{\mu},k}\|_{L^1\left(\mathbb{R}\right)}=1.
\end{equation}
For the sake of brevity, in what follows we shall omit $k$ from $v_{k}$.

In order to prove that $v(\cdot,y)\in  H^2\left(B_2\right) \cap H^1\left(B_2\right)$ for $y\in \mathbb{R}$, let $M,N\in\mathbb{N}$ such that $M>N$ and let us denote by

\begin{equation}
\label{3-83}
V_{M,N}(x,y):=\sum_{j=N+1}^{M}\alpha_j \widehat{\varphi}_{\overline{\mu},k}\left(\sqrt{\lambda_j}\right)g_k\left(y\sqrt{\lambda_j}\right)e_j(x).
\end{equation}
By \eqref{2-80} and \eqref{1-83} we have, for every $y\in\mathbb{R}$,

\begin{gather*}
 \lambda\int_{B_2}\left\vert\nabla_x V_{M,N}(x,y)\right\vert^2 dx\leq \int_{B_2}A(x)\nabla_x V_{M,N}(x,y)\cdot\nabla_x V_{M,N}(x,y) dx=\\
=\sum_{j=N+1}^{M} \left(\int_{B_2}A(x)\nabla_x e_j(x)\cdot\nabla_x V_{M,N}(x,y) dx\right)\widehat{\varphi}_{\overline{\mu},k}\left(\sqrt{\lambda_j}\right)g_k\left(y\sqrt{\lambda_j}\right)\alpha_j=\\
=\sum_{j=N+1}^{M}\lambda_j\alpha^2_j\widehat{\varphi}^2_{\overline{\mu},k}\left(\sqrt{\lambda_j}\right) g^2_k\left(y\sqrt{\lambda_j}\right)\leq c e^{4k}\sum_{j=N+1}^{M}\lambda_j\alpha^2_j.
\end{gather*}
Therefore, since $V_{M,N}(\cdot,y)\in H^1_0\left(B_2\right)$ we have
\begin{equation}
\label{84}
 \|V_{M,N}(\cdot,y)\|^2_{H^1_0\left(B_2\right)}\leq c e^{4k}\sum_{j=N+1}^{M}\lambda_j\alpha^2_j \mbox{, for every } y\in \mathbb{R}.
\end{equation}
The inequality above and \eqref{1-72} gives
\begin{equation*}
 \|V_{M,N}(\cdot,y)\|_{H^1_0\left(B_2\right)}\rightarrow 0 \mbox{, as } M,N\rightarrow \infty \mbox{, for every } y\in \mathbb{R},
\end{equation*}
hence $v\in H^1_0\left(B_2\right)$.

In order to prove that $v\in H^2\left(B_2\right)$, first observe that by \eqref{2-80}, \eqref{1-83} \eqref{3-83} we have
\begin{equation*}
 \|\mbox{div}\left(A\nabla_x V_{M,N}\right)\|^2_{L^2\left(B_2\right)}\leq c\lambda^{-1} e^{4k}\sum_{j=N+1}^{M}\lambda^2_j\alpha^2_j \mbox{, for every } y\in \mathbb{R},
\end{equation*}
then by the above inequality and standard $L^2$ regularity estimate \cite{GT} we obtain
\begin{gather}
\label{2-85}
 \|D^2_x V_{M,N}(\cdot,y)\|^2_{L^2\left(B_2\right)}\leq \\ \nonumber
 \leq C \|\mbox{div}\left(A\nabla_x V_{M,N}\right)\|^2_{L^2\left(B_2\right)}\leq e^{4k}\sum_{j=N+1}^{M}\lambda^2_j\alpha^2_j \mbox{, for every } y\in \mathbb{R},
\end{gather}
where $C$ depends on $\lambda$ and $\Lambda$ only. Hence $v\in H^2\left(B_2\right)$. Moreover by \eqref{1-72}, \eqref{84} and \eqref{2-85} we have
 \begin{gather}
\label{2-86}
 \|v(\cdot,y)\|_{L^2\left(B_2\right)}+\|\nabla_x v(\cdot,y)\|_{L^2\left(B_2\right)}+\|D^2_x v(\cdot,y)\|_{L^2\left(B_2\right)}\leq \\ \nonumber \leq C H e^{2k} \mbox{, for every } y\in \mathbb{R},
\end{gather}
where $C$ depends on $\lambda$ and $\Lambda$ only.

Similarly we have $\partial_yv (\cdot,y) ,\partial^2_y v (\cdot,y), \partial_y \nabla_x v(\cdot,y)\in L^2\left(B_2\right)$ and
\begin{equation}
\label{86}
\sum_{j=1}^{2}\|\partial^j_y D^{2-j}_x v(\cdot,y)\|_{L^2\left(B_2\right)}\leq CH e^{2k} \mbox{, for every  } y\in \mathbb{R},
\end{equation}
where $C$ depends on $\lambda$ and $\Lambda$ only.
Inequality \eqref{86} and \eqref{2-86}, yields \eqref{3-81}.

By \eqref{3-80} we have immediately that the function $v$ is an even function and it satisfies \eqref{4-5-6-81}.
Concerning \eqref{2-81}, first observe that by the definition of $g_k$ we have that $g''_k (y\sqrt{\lambda_j})-g_k(y\sqrt{\lambda_j})=0$, for $|y|\sqrt{\lambda_j}\leq k$ and $\left\vert g''_k(y\sqrt{\lambda_j})-g_k(y\sqrt{\lambda_j}) \right\vert\leq ce^{2k}$, for $|y|\sqrt{\lambda_j}\geq k$. Hence, taking into account \eqref{3'-75} and \eqref{fourierthetak}, we have, for every $y\in\mathbb{R}$ and for every $k\in\mathbb{N}$,

\begin{gather}
\label{90}
\left\vert g''_k(y\sqrt{\lambda_j})-g_k(y\sqrt{\lambda_j}) \right\vert \left\vert \widehat{\varphi}_{\overline{\mu}, k}(\sqrt{\lambda_j}) \right\vert \leq \\ \nonumber
\leq c e^{2k}\left\vert \widehat{\vartheta}(k^{-1}\sqrt{\lambda_j}) \right\vert^k \chi_{\{y:|y|\sqrt{\lambda_j}\geq k\}}\leq
 \\ \nonumber
  \leq c e^{2k}\sup \left\{\left\vert \widehat{\vartheta}(k^{-1}\sqrt{\lambda_j}) \right\vert^k : |y|\sqrt{\lambda_j}\geq k\right\}\leq c e^{2k} \min\left\{1,\left(4\pi\lambda^{-1}|y|\right)^{2k}\right\}.
\end{gather}
By \eqref{1-81} and  \eqref{90} we have
\begin{equation*}
 \|f_{k}(\cdot,y)\|_{L^2\left(B_2\right)}\leq c e^{2k} \min\left\{1,\left(4\sqrt{2}\pi\lambda^{-1}|y|\right)^{2k}\right\}\left(\sum_{j=1}^{\infty}\lambda^2_j\alpha^2_j\right)^{1/2} \mbox{, for every } y\in \mathbb{R}.
\end{equation*}
By the above inequality and by \eqref{1-72} we obtain \eqref{2-81}.

Since $\|\varphi_{\overline{\mu},k}\|_{L^1\left(\mathbb{R}\right)}=1$, by Schwarz inequality and by \eqref{4ii-65}  and \eqref{2-76} we have

\begin{gather*}
\|v_k(\cdot,0)\|^2_{L^2\left(B_{r_0}\right)}=\int_{B_{r_0}} \left \vert \widetilde{u}_{k}(x)\right \vert ^{2}dx\leq\\
\leq \int_{-\lambda \left( \overline{\mu} +1\right) /4}^{\lambda \left( \overline{\mu} +1\right)
/4}\left(\int_{B_{r_0}}\left \vert u(x,t)\right \vert ^{2}dx\right)\varphi_{\overline{\mu},k}(t)dt\leq \varepsilon^2
\end{gather*}
and \eqref{1-82} follows.
\end{proof}

In what follows we shall denote by $\widetilde{B}_r$ the ball of $\mathbb{R}$ of radius $r$ centered at $0$.

In order to prove Proposition \ref{3-91} stated below we need the following theorem that has been proved in \cite[Theorem 1.10]{A-R-R-V}

\begin{theo}\label{stimacauchy}
Let $r$ be a positive number and let $w\in H^2\left(B_r\right)$ be a solution to the problem
\begin{equation}
\label{pbCauchy}
\left\{\begin{array}{ll}
q(x)\partial^2_{y}w(x,y)+\mbox{div}\left(A(x)\nabla_x w(x,y)\right)=0, \quad \hbox{ in } \widetilde{B}_r, \\[2mm]
\partial_{y}w(\cdot,0)=0 , \quad \hbox{in } B_r,
\end{array}\right.
\end{equation}
where $A$  satisfies \eqref{1-65} and $q$ satisfies \eqref{3-65}.

Then there exist $\beta\in (0,1)$ and $C\geq 1$ depending on $\lambda$ and $\Lambda$ only such that
\begin{equation}
\label{stimaCauchy}
\int_{\widetilde{B}_{r/4}}w^2dxdy\leq C\left(\int_{\widetilde{B}_{r}}w^2dxdy\right)^{1-\beta}\left(\int_{B_{r/2}}w^2(x,0)dx\right)^{\beta}.
\end{equation}
\end{theo}

\bigskip

\begin{prop}\label{3-91}
Let $v_{k}$ be defined in \eqref{3-80} and let $r_0\leq \frac{\lambda}{8}$. Then we have
\begin{equation}
\label{1-91}
\|v_{k}\|_{L^2\left(\widetilde{B}_{r_0/4}\right)}\leq C\left(\varepsilon+H\left(C_0r_0\right)^{2k}\right)^{\beta}\left(He^{2k}+H\left(C_0r_0\right)^{2k}\right)^{1-\beta}.
\end{equation}
where $\beta\in (0,1)$, $C$ depend on $\lambda$ and $\Lambda$ only and $C_0=4\pi e \lambda^{-1}$.
\end{prop}

\begin{proof}
Let $w_{k}\in H^2\left(\widetilde{B}_{r_0}\right)$ be the solution to the following Dirichlet pronlem
\begin{equation}
\label{1-92}
\left\{\begin{array}{ll}
q(x)\partial^2_{y}w_{k}+\mbox{div}\left(A(x)\nabla_x w_{k}\right)=f_{k}, \quad \hbox{ in } \widetilde{B}_{r_0},\\[2mm]
w_{k}=0, \quad \hbox{on } \partial \widetilde{B}_{r_0}.
\end{array}\right.
\end{equation}
Notice that, since $f_{k}$ is an even function with respect to $y$, by the uniqueness to the Dirichlet problem \eqref{1-92} we have that $w_{k}$ is an even function with respect to $y$.

By standard regularity estimates \cite{GT} we have

\begin{equation}
\label{2-92}
\|w_{k}\|_{L^2\left(\widetilde{B}_{r_0}\right)}+r_0\|\nabla_{x,y}w_{k}\|_{L^2\left(\widetilde{B}_{r_0}\right)}\leq C\|f_{k}\|_{L^2\left(\widetilde{B}_{r_0}\right)},
\end{equation}
where $C$ depends on $\lambda$ only.
By the above inequality and by the trace inequality we get
\begin{gather}
\label{1-93}
\|w_{k}(\cdot,0)\|_{L^2\left(B_{r_0/2}\right)}\leq \\ \nonumber
\leq C\left(r^{-1/2}_0\|w_{k}\|_{L^2\left(\widetilde{B}_{r_0}\right)}+r^{1/2}_0\|\nabla_{x,y}w_{k}\|_{L^2\left(\widetilde{B}_{r_0}\right)}\right)
\leq C r^{3/2}_0\|f_{k}\|_{L^2\left(\widetilde{B}_{r_0}\right)},
\end{gather}
where $C$ depends on $\lambda$ only.

Now, denoting by

\begin{equation}
\label{2-93}
z_{k}=v_{k}-w_{k},
\end{equation}
by \eqref{2-81}, \eqref{1-82}, \eqref{2-92} and \eqref{1-93} we have
\begin{equation}
\label{1-94}
\|z_{k}(\cdot,0)\|_{L^2\left(B_{r_0/2}\right)}\leq \varepsilon+Cr^2_0H\left(C_0r_0\right)^{2k},
\end{equation}
and

\begin{equation}
\label{2-94}
\|z_{k}\|_{L^2\left(\widetilde{B}_{r_0}\right)}\leq Cr^{1/2}_0H\left(e^{2k}+r^2_0\left(C_0r_0\right)^{2k}\right),
\end{equation}
where $C$ depends on $\lambda$ only.

Now by \eqref{1-92} we have

\begin{equation*}
\left\{\begin{array}{ll}
q(x)\partial^2_{y}z_{k}+\mbox{div}\left(A(x)\nabla_x z_{k}\right)=0, \quad \hbox{ in } \widetilde{B}_{r_0},\\[2mm]
\partial_y z_{k}(\cdot,0)=0, \quad \hbox{on } B_{r_0},
\end{array}\right.
\end{equation*}
hence by applying Theorem \ref{stimacauchy} to the function $z_{k}$ and by using \eqref{3-81},  \eqref{2-93}, \eqref{1-94} and \eqref{2-94}  the thesis follows.
\end{proof}
In order to prove Theorem \ref{5-115} we use a Carleman estimate proved, in the context of parabolic operator, in  \cite{EsVe}.

Let $P$ be the elliptic operator
\begin{equation}
\label{oper}
P:=q(x)\partial^2_{y}+\mbox{div}\left(A(x)\nabla_x \right),
\end{equation}

\begin{theo}\label{Carleman}
Let $P$ be the operator \eqref{oper} and assume that \eqref{1-65} and \eqref{3-65} are satisfied. There exists a constant $C_{\ast}>1$ depending on $\lambda$ and $\Lambda$ only such that, denoting

\begin{subequations}
\label{4-5-6-96}
\begin{equation}
\label{4-96}
\phi(s)=s\exp \left(\int^s_0\frac{e^{-C\eta}-1}{\eta}d\eta\right),
\end{equation}
\begin{equation}
\label{6-96}
\sigma(x,y)=\left(A^{-1}(0)x\cdot x+\left(q(0)\right)^{-1}y^2\right)^{1/2},
\end{equation}
\begin{equation}
\label{5-96}
\delta(x,y)=\phi\left(\sigma(x,y)/2\sqrt{\lambda}\right),
\end{equation}
\begin{equation}
\label{palla n+1}
\widetilde{B}^{\sigma}_{r}= \left\{(x,y)\in \mathbb{R}^{n+1}: \sigma(x,y)\leq r\right\}, \quad \hbox{  } r>0,
\end{equation}
\end{subequations}
for every $\tau\geq C_{\ast}$ and $U\in C^{\infty}_0\left(\widetilde{B}^{\sigma}_{2\sqrt{\lambda}/C_{\ast}}\setminus\{0\}\right)$ we have

\begin{gather}
\label{6-118}
\tau\int_{\mathbb{R}^{n+1}}\delta^{1-2\tau}(x,y)\left\vert\nabla_{x,y}U\right\vert^2 dxdy+
\tau^3\int_{{\mathbb{R}^{n+1}}}\delta^{-1-2\tau}(x,y)\left\vert U\right\vert^2 dxdy\leq\\ \nonumber
\leq C_{\ast}\int_{{\mathbb{R}^{n+1}}}\delta^{2-2\tau}(x,y) \left\vert PU\right\vert^2dxdy.
\end{gather}

\end{theo}

\bigskip

\textbf{Conclusion of the proof of Theorem \ref{5-115}}

We begin to observe that
\begin{equation}
\label{5-119}
\widetilde{B}^{\sigma}_{\sqrt{\lambda} r}\subset \widetilde{B}_r \subset \widetilde{B}^{\sigma}_{r/\sqrt{\lambda}} \quad \hbox{, for every   }r>0,
\end{equation}
and setting $r_1=\frac{\sqrt{\lambda} r_0}{16}$, by \eqref{1-91} we have

\begin{equation}
\label{2-120}
\|v_{k}\|_{L^2\left(\widetilde{B}^{\sigma}_{4r_1}\right)}\leq C S_k,
\end{equation}
where $C$ depends on $\lambda$ and $\Lambda$ only and
\begin{equation}
\label{3-120}
 S_k=\left(\varepsilon+H\left(C_1r_1\right)^{2k}\right)^{\beta}\left(He^{2k}+H\left(C_1r_1\right)^{2k}\right)^{1-\beta},
\end{equation}
where $C_1=16C_0/\sqrt{\lambda}$.

Denote by
\[\delta_0(r):=\phi(r/2\sqrt{\lambda})\quad\hbox{, for every } r>0\]
and let us consider a function $h\in C^2_0\left(0, \delta_0\left(2\sqrt{\lambda}/C_{\ast}\right)\right)$ such that $0\leq h\leq 1$ and

\begin{subequations}
\begin{equation*}
h(s)=1 ,\quad\hbox{ for every   } s\in\left[\delta_0\left(2r_1\right), \delta_0\left(\sqrt{\lambda}/C_{\ast}\right)\right],
\end{equation*}
\begin{equation*}
h(s)=0, \quad\hbox{ for every   } s\in\left[0,\delta_0\left(r_1\right)\right]\cup \left[\delta_0\left(3\sqrt{\lambda}/2C_{\ast}\right), \delta_0\left(2\sqrt{\lambda}/C_{\ast}\right)\right],
\end{equation*}
\begin{equation*}
r_1\left\vert h'(s)\right\vert+r_1^2\left\vert h''(s)\right\vert\leq c, \quad\hbox{ for every } s\in\left[\delta_0\left(r_1\right), \delta_0\left(2r_1\right)\right],
\end{equation*}
\begin{equation*}
\left\vert h'(s)\right\vert +\left\vert h''(s)\right\vert\leq c, \quad\hbox{ for every } s\in\left[\delta_0\left(\sqrt{\lambda}/C_{\ast}\right), \delta_0\left(3\sqrt{\lambda}/2C_{\ast}\right)\right],
\end{equation*}
\end{subequations}
where $c$ is an absolute constant.

Moreover, let us define
\begin{equation*}
\zeta(x,y)=h\left(\delta(x,y)\right).
\end{equation*}
Notice that if $2r_1\leq \sigma(x,y)\leq \sqrt{\lambda}/C_{\ast}$ then $\zeta(x,y)=1$ and if $\sigma(x,y)\geq 2\sqrt{\lambda}/C_{\ast}$ or $\sigma(x,y)\leq r_1$ then $\zeta(x,y)=0$.

For the sake of brevity, in what follows we shall omit $k$ from $v_{k}$ and $f_k$. By density, we can apply \eqref{6-118} to the function $U=\zeta v$ and we have, for every $\tau\geq C_{\ast}$,

\begin{gather}
\label{1-121}
\tau\int_{\widetilde{B}^{\sigma}_{2\sqrt{\lambda}/C_{\ast}}}\delta^{1-2\tau}(x,y)\left\vert\nabla_{x,y}\left(\zeta v\right)\right\vert^2+
\tau^3\int_{\widetilde{B}^{\sigma}_{2\sqrt{\lambda}/C_{\ast}}}\delta^{-1-2\tau}(x,y) \left\vert\zeta v\right\vert^2 \leq\\ \nonumber
\leq C\int_{\widetilde{B}^{\sigma}_{2\sqrt{\lambda}/C_{\ast}}}\delta^{2-2\tau}(x,y) \left\vert f\right\vert^2 \zeta^2+C\int_{\widetilde{B}^{\sigma}_{2\sqrt{\lambda}/C_{\ast}}}\delta^{2-2\tau}(x,y) \left\vert P\zeta\right\vert^2 v^2+ \\ \nonumber
+C\int_{\widetilde{B}^{\sigma}_{2\sqrt{\lambda}/C_{\ast}}}\delta^{2-2\tau}(x,y) \left\vert\nabla_{x,y} v\right\vert^2\left\vert\nabla_{x,y}\zeta\right\vert^2:=I_1+I_2+I_3 ,
\end{gather}
where $C$ depends $\lambda$ and $\Lambda$ only.

\bigskip

\textit{Estimate of $I_1$}.

Notice that

\begin{equation}
\label{2-122}
\frac{\sqrt{|x|^2+y^2}}{2C_2}\leq\delta(x,y)\leq\frac{C_2\sqrt{|x|^2+y^2}}{2},
\end{equation}
where $C_2>1$ depends on $\lambda$ and $\Lambda$ only.

By \eqref{1-81}, \eqref{5-119} and \eqref{2-122} we have

\begin{gather}
\label{I-1}
\int_{\widetilde{B}^{\sigma}_{2\sqrt{\lambda}/C_{\ast}}}\delta^{2-2\tau}(x,y) \left\vert f\right\vert^2 \zeta^2dxdy\leq \int_{\widetilde{B}_{2}}\left(2C_2 |y|^{-1}\right)^{-2+2\tau} \left\vert f\right\vert^2 dxdy\leq\\ \nonumber
\leq\int^{2}_{-2}\left[\left(2C_2 |y|^{-1}\right)^{-2+2\tau}\int_{B_2}\left\vert f(x,y)\right\vert^2 dx\right]dy\leq CH^2 \int^{2}_{-2}\left(2C_2 |y|^{-1}\right)^{-2+2\tau}\left(C_0|y|\right)^{4k}dy,
\end{gather}
where $C$ depends on $\lambda$ and $\Lambda$ only.

Now let $k$ and $\tau$ satisfy the relation

\begin{equation}
\label{1-123}
\frac{\tau-1}{2}\leq k.
\end{equation}
By \eqref{I-1} and \eqref{1-123} we get

\begin{equation}
\label{2-123}
I_1\leq C H^2 \left(C_3\right)^{4k},
\end{equation}
where $C_3=2 C_0 C_2$.

\bigskip

\textit{Estimate of $I_2$}

By \eqref{3-81} and \eqref{2-120} and \eqref{1-121}  we have
\begin{gather*}
I_2 \leq Cr_1^{-4}\int_{\widetilde{B}^{\sigma}_{2r_1}\setminus\widetilde{B}^{\sigma}_{r_1}}\delta^{2-2\tau}(x,y) v^2 dxdy
+C\int_{\widetilde{B}^{\sigma}_{3\sqrt{\lambda}/2C_{\ast}}\setminus\widetilde{B}^{\sigma}_{\sqrt{\lambda}/C_{\ast}}}\delta^{2-2\tau}(x,y) v^2 dxdy \leq\\
\leq C\left(r_1^{-3}\delta^{2-2\tau}_0(r_1)S^2_k+e^{4k}H^2\delta^{2-2\tau}_0(\sqrt{\lambda}/C_{\ast})\right),
\end{gather*}
hence \eqref{2-122} gives

\begin{equation}\label{1-125}
I_2 \leq C\left(\delta^{-1-2\tau}_0(r_1)S^2_k+e^{4k}H^2\delta^{-1-2\tau}_0(\sqrt{\lambda}/C_{\ast})\right).
\end{equation}

\textit{Estimate of $I_3$}

By \eqref{1-121}  we have

\begin{gather}\label{2-125}
I_3 \leq Cr_1^{-2}\delta^{2-2\tau}_0(r_1)\int_{\widetilde{B}^{\sigma}_{2r_1}\setminus\widetilde{B}^{\sigma}_{r_1}} \left\vert\nabla_{x,y} v\right\vert^2 dxdy+\\ \nonumber+C\delta^{2-2\tau}_0(\sqrt{\lambda}/C_{\ast})\int_{\widetilde{B}^{\sigma}_{3\sqrt{\lambda}/2C_{\ast}}\setminus\widetilde{B}^{\sigma}_{\sqrt{\lambda}/C_{\ast}}} \left\vert\nabla_{x,y} v\right\vert^2 dxdy.
\end{gather}
Now in order to estimate from above the righthand side of \eqref{2-125} we use the Caccioppoli inequality, \eqref{3-81}, \eqref{2-81} and \eqref{2-120} and we get
\begin{gather}
\label{1-127}
I_3 \leq C\delta^{2-2\tau}_0(r_1)\left(r_1^{-4}\int_{\widetilde{B}^{\sigma}_{4r_1}\setminus\widetilde{B}^{\sigma}_{r_1/2}} v^2 dxdy+\int_{\widetilde{B}^{\sigma}_{4r_1}\setminus\widetilde{B}^{\sigma}_{r_1/2}} f^2 dxdy\right)+\\ \nonumber+C\delta^{2-2\tau}_0(\sqrt{\lambda}/C_{\ast})\int_{\widetilde{B}^{\sigma}_{3\sqrt{\lambda}/2C_{\ast}}\setminus\widetilde{B}^{\sigma}_{\sqrt{\lambda}/C_{\ast}}} \left\vert\nabla_{x,y} v\right\vert^2 dxdy \leq\\ \nonumber \leq C  \left(S_k^2+H^2\left(C_1r_1\right)^{4k}\right)\delta^{1-2\tau}_0(r_1)+CH^2e^{4k}\delta^{1-2\tau}_0(\sqrt{\lambda}/C_{\ast}):=\widetilde{I}_3.
\end{gather}

\bigskip

Now let $r_1\leq \frac{\sqrt{\lambda}}{2C_{\ast}}$, let $\rho$ be such that $\frac{2r_1}{\sqrt{\lambda}}\leq\rho\leq\frac{1}{C_{\ast}}$ and denote by $\widetilde{\rho}=\sqrt{\lambda}\rho$. By estimating from below trivially the left hand side of \eqref{1-121} and taking into account \eqref{1-127} we have
\begin{equation}
\label{2-127}
\delta^{1-2\tau}_0(\widetilde{\rho})\int_{\widetilde{B}^{\sigma}_{\widetilde{\rho}}\setminus\widetilde{B}^{\sigma}_{2r_1}}\left\vert\nabla_{x,y} v\right\vert^2+
\delta^{-1-2\tau}_0(\widetilde{\rho})\int_{\widetilde{B}^{\sigma}_{\widetilde{\rho}}\setminus\widetilde{B}^{\sigma}_{2r_1}}\left\vert v\right\vert^2\leq I_1+I_2+\widetilde{I}_3.
\end{equation}
Now let us add at both the side of \eqref{2-127} the quantity
\begin{equation*}
\delta^{1-2\tau}_0(\widetilde{\rho})\int_{\widetilde{B}^{\sigma}_{2r_1}}\left\vert\nabla_{x,y} v\right\vert^2+
\delta^{-1-2\tau}_0(\widetilde{\rho})\int_{\widetilde{B}^{\sigma}_{2r_1}} v^2,
\end{equation*}
by using standard estimates for second order elliptic equations and by taking into account that $\delta_0(\widetilde{\rho})\geq\delta_0(r_1)$, we have
\begin{equation}
\label{1-128}
\int_{\widetilde{B}^{\sigma}_{\widetilde{\rho}}}\left\vert\nabla_{x,y} v\right\vert^2+
\int_{\widetilde{B}^{\sigma}_{\widetilde{\rho}}} v^2 \leq \delta^{1+2\tau}_0(\widetilde{\rho}) \left(I_1+I_2+C\widetilde{I}_3\right),
\end{equation}
where $C$ depends on $\lambda$ and $\Lambda$ only.

Now by \eqref{2-122}, \eqref{2-123}, \eqref{1-125}, \eqref{1-127} and \eqref{1-128} it is simple to derive that if \eqref{1-123} is satisfied then
we have
\begin{gather}
\label{3-130}
\int_{\widetilde{B}_{\lambda\rho}}\left\vert\nabla_{x,y} v\right\vert^2+
\int_{\widetilde{B}_{\lambda\rho}} v^2 \leq\\ \nonumber \leq C \left[S^2_k\left(\frac{\delta_0(\widetilde{\rho})}{\delta_0(r_1)}\right)^{1+2\tau}+H^2C_4^k\left(\frac{\delta_0(\widetilde{\rho})}
{\delta_0(\sqrt{\lambda}/C_{\ast})}\right)^{1+2\tau}\right],
\end{gather}
where $C_4>1$ depends on $\lambda$ and $\Lambda$ only.

Now, by applying a standard trace inequality and by recalling that $v(\cdot,0)=\widetilde{u}_{k}(\cdot,0)$ in $B_2$ (where $\widetilde{u}_{k}$ is defined by \eqref{u-k}) we have

\begin{gather}
\label{1-131}
\int_{B_{\lambda\rho/2}} \left\vert\widetilde{u}_{k}(\cdot,0) \right\vert^2\leq \\ \nonumber \leq C\rho^{-1}\left[S^2_k\left(\frac{\delta_0(\widetilde{\rho})}{\delta_0(r_1)}\right)^{1+2\tau}+H^2C_4^k\left(\frac{\delta_0(\widetilde{\rho})}
{\delta_0(\sqrt{\lambda}/C_{\ast})}\right)^{1+2\tau}\right].
\end{gather}

By Proposition \ref{4-97}, by \eqref{3-120} and \eqref{1-131} we have, for $r_1\leq \frac{\sqrt{\lambda}}{2C_{\ast}}$

\begin{gather}
\label{1-132}
\rho\int_{B_{\lambda\rho/2}} \left\vert u(\cdot,0) \right\vert^2\leq C \left(H_{k,\tau}+H^2k^{-1/6}\right)+\\ \nonumber +C \left[C_5^{k}\left(\frac{\delta_0(\widetilde{\rho})}{\delta_0(r_1)}\right)^{1+2\tau}H^{2(1-\beta)}
\varepsilon^{2\beta}+H^2C_4^k\left(\frac{\delta_0(\widetilde{\rho})}
{\delta_0(\sqrt{\lambda}/C_{\ast})}\right)^{1+2\tau} \right],
\end{gather}
where

\begin{equation*}
H_{k,\tau}:=H^2 \left(\frac{\delta_0(\widetilde{\rho})}{\delta_0(r_1)}\right)^{1+2\tau}C_5^{k}r_1^{4\beta k}.
\end{equation*}
and $C$, $C_5$ depend on $\lambda, \Lambda$ only.

Now let us choose $\tau=\frac{4\beta k-1}{2}$. We have that \eqref{1-123} is satisfied and
by \eqref{2-122}, \eqref{1-132} we have that there exist constants $C_6>1$ and $k_0$ depending on $\lambda$ and $\Lambda$ only such that for every $k\geq k_0$ we have

\begin{gather}
\label{1-135}
\rho\int_{B_{\lambda\rho/2}} \left\vert u(\cdot,0) \right\vert^2\leq C_6 H_1^2\left[\left(C_6\rho r_1^{-1}\right)^{4\beta k}\varepsilon_1^{2\beta}+\left(C_6\rho\right)^{4\beta k}+k^{-1/6}\right],
\end{gather}
where
\begin{equation*}
H_1:=H+e\varepsilon \quad \mbox{ and } \varepsilon_1:=\frac{\varepsilon}{H+e\varepsilon}.
\end{equation*}

Now, let us denote by
\[\overline{k}:= \left[\frac{\log \varepsilon_1}{2\log r_1}\right]+1,\]
where, for any $s\in\mathbb{R}$, we set $[s]:=\max\left\{p\in\mathbb{Z}:p \leq s\right\}$. If $\overline{k}\leq k_0$ we choose $k=\overline{k}$ so that by \eqref{1-135} we have, for $\rho\leq 1/C_6$,
\begin{equation}
\label{1-136}
\rho\int_{B_{\lambda\rho/2}} \left\vert u(\cdot,0) \right\vert^2\leq C_2 H_1^2\left(\varepsilon_1^{2\beta\theta_0}+\left(\frac{2\log (1/r_1)}{\log (1/\varepsilon_1)}\right)^{1/6}\right),
\end{equation}
where
\begin{equation}
\label{2-136}
\theta_0=\frac{\log (1/C_6\rho)}{2\log (1/r_1)}.
\end{equation}
Otherwise, if $\overline{k} < k_0$ then multiplying both the side of such an inequality by $\log (1/C_6\rho)$ and by \eqref{2-136} we get $\theta_0 \log (1/\varepsilon_1)\leq k_0 \log (1/C_6\rho)$. Hence
\[(H+e\varepsilon)^{2\beta\theta_0}\leq (C_6\rho)^{-2\beta k_0}\varepsilon^{2\beta\theta_0}.\]
By this inequality and by \eqref{4iii-65} we have trivially
\begin{gather}
\label{2-137}
\int_{B_{\lambda\rho/2}} \left\vert u(\cdot,0) \right\vert^2\leq  (H+e\varepsilon)^2=\\ \nonumber (H+e\varepsilon)^{2(1-\beta\theta_0)}\varepsilon^{2\beta\theta_0}\leq (C_1\rho)^{-2\beta k_0}\varepsilon^{2\beta\theta_0}.
\end{gather}
Finally by \eqref{1-136} and \eqref{2-137} we obtain \eqref{SUCP}.
$\Box$

\subsection{Proof of Theorem \ref{5-115Boundary}}\label{SUCP boundary}
First, let us assume $A(0)=I$ where $I$ is the identity matrix $n\times n$. Following the arguments of \cite{AE} or  \cite{A-B-R-V} we have there exist $\rho_1, \rho_2\in (0,\rho_0]$ such that $\frac{\rho_1}{\rho_0},\frac{\rho_2}{\rho_0}$ depend on $\lambda,\Lambda, E$ only and we can construct a function $\Phi\in C^{1,1}(\overline{B}_{\rho_2}(0),\mathbb{R}^n)$ such that

\begin{subequations}
\label{Phi}
\begin{equation}
\label{Phia}
\Phi\left(B_{\rho_2}\right)\subset B_{\rho_1},
\end{equation}
\begin{equation}
\label{Phib}
\Phi(y',0)=(y',\phi(y')),\quad\hbox{ for every }  y'\in B^{\prime}_{\rho_2},
\end{equation}
\begin{equation}
\label{Phic}
\Phi\left(B^+_{\rho_2}\right)\subset K_{\rho_1},
\end{equation}
\begin{equation}
\label{Phid}
C_1^{-1} |y-z|\leq |\Phi(x)-\Phi(z)|\leq C_1|y-z|,\quad\hbox{ for every }  y,z\in B_{\rho_2},
\end{equation}
\begin{equation}
\label{Phie}
C_2^{-1}\leq |\textrm{det}D\Phi(y)|\leq C_2,\quad\hbox{ for every }  y\in B_{\rho_2},
\end{equation}
\begin{equation}
\label{Phif}
|\textrm{det}D\Phi(y)-\textrm{det}D\Phi(z)|\leq C_3|y-z|,\quad\hbox{ for every }  y,z\in B_{\rho_2},
\end{equation}
\end{subequations}
where $C_1,C_2,C_3\geq 1$ depend on $\lambda,\Lambda, E$ only.

Denoting

\begin{equation*}
\overline{A}(y)=|\textrm{det}D\Phi(y)|(D\Phi^{-1})(\Phi(y)) A(\Phi(y))(D\Phi^{-1})^{\ast}(\Phi(y)),
\end{equation*}
\begin{equation}
v(y,t)=u(\Phi(y),t),
\end{equation}
we have
\begin{subequations}
\label{A}
\begin{equation}
\label{Aa}
\overline{A}(0)=I,
\end{equation}
\begin{equation}
\label{Ab}
\overline{a}^{nk}(y',0)=\overline{a}^{kn}(y',0)=0, k=1,\ldots,n-1.
\end{equation}
\end{subequations}

Moreover, we have that the ellipticity and Lipschitz constants of $\overline{A}$ depend on $\lambda,\Lambda, E$ only. For every $y\in B_{\rho_2}(0)$, let us denote by $\tilde A(y)=\{\tilde a_{ij}(y)\}_{i,j=1}^n$ the matrix with entries given by
\[
\tilde a^{ij}(y',|y_n|)=\overline{a}^{ij}(y',|y_n|),\quad\hbox{ if either }  i,j\in\{1,\ldots,n-1\}\hbox{, or }i=j=n,
\]
\[
\tilde a^{nj}(y',y_n)=\tilde a^{jn}(y',y_n)=\textrm{sgn}(y_n)\overline{a}^{nj}(y',|y_n|),\quad\hbox{ if }  1\leq j\leq n-1.
\]
We have that $\tilde A$ satisfies the same ellipticity and Lipschitz continuity conditions as $\overline{A}$.

Now, if $u$ satisfies the boundary condition \eqref{DirichBoundary} then we define
\begin{equation*}
U(y,t)=\textrm{sgn}(y_n)v(y',|y_n|,t),\quad\hbox{ for } (y,t)\in B_{\rho_2}\times (-\lambda\rho_2,\lambda\rho_2),
\end{equation*}

\begin{equation*}
\widetilde{q}(y)=|\textrm{det}D\Phi(y',|y_n|)|,\quad\hbox{ for } y\in B_{\rho_2},
\end{equation*}
we have that $U\in\mathcal{W}\left((-\lambda\rho_2,\lambda\rho_2);B_{\rho_2}\right)$ is a solution to
\begin{equation} \label{4i-65Boundary1}
\widetilde{q}(y)\partial^2_{t}U-\mbox{div}\left(\widetilde{A}(y)\nabla U\right)=0, \quad \hbox{in } B_{\rho_2}\times(-\lambda\rho_2,\lambda\rho_2).
\end{equation}
Moreover, by \eqref{Phid} we have that
\begin{equation*}
K_{r/C_1}\subset\Phi\left(B^+_{r}\right)\subset K_{C_1r}\quad\hbox{, for every } r\leq\rho_2.
\end{equation*}
Now we can apply Theorem \ref{5-115} to the function $U$ and then by simple changes of variables in the integrals we obtain \eqref{SUCPBoundary}. In the general case $A(0)\neq I$ we can consider a linear transformation $G:\mathbb{R}^n\rightarrow \mathbb{R}^n$ such that setting $A'(Gx)=\frac{GA(x)G^{\ast}}{\textrm{det}G}$ we have $A'(0)=I$. Therefore, noticing that

\begin{equation*}
B_{\sqrt{\lambda} r}\subset G\left(B_{r}\right)\subset B_{\sqrt{\lambda^{-1}} r},\quad\hbox{ for every } r>0,
\end{equation*}
it is a simple matter to get \eqref{SUCPBoundary} in the general case.

If $u$ satisfies the boundary condition \eqref{NeumBoundary} then we define
\begin{equation*}
V(y,t)=v(y',|y_n|,t),\quad\hbox{ for } (y,t)\in B_{\rho_2}\times (-\lambda\rho_2,\lambda\rho_2),
\end{equation*}
and we get that $V$ is a solution to \eqref{4i-65Boundary}. Therefore, arguing as before we obtain again \eqref{SUCPBoundary}.$\square$

\section{Concluding Remark - A first order perturbation}\label{Concluding}
In this subsection we outline the proof of an extension of Theorems \ref{5-115}, \ref{5-115Boundary} for solution to the equation
\begin{equation} \label{4i-65CR}
q(x)\partial^2_{t}u-Lu=0, \quad \hbox{in } B_{\rho_0}\times(-\lambda\rho_0,\lambda\rho_0).
\end{equation}
where
\begin{equation} \label{L}
Lu=\mbox{div}\left(A(x)\nabla_x u\right)+b(x)\cdot\nabla_x u+c(x)u,
\end{equation}
and $A$, $q$ satisfy \eqref{1-65}, \eqref{3-65}, $b=(b^1,\cdots,b^n)$ $b^j\in C^{0,1}(\mathbb{R}^n)$, $c\in L^{\infty}(\mathbb{R}^n)$. Moreover we assume

\begin{subequations}
\begin{equation} \label{3-65aCR}
\left\vert b(x)\right\vert\leq\lambda^{-1}\rho_0^{-1}, \quad \hbox{for every } x\in\mathbb{R}^n,
\end{equation}
\begin{equation} \label{3'-65CR}
\left\vert b(x)-b(y)\right\vert\leq\frac{\Lambda}{\rho^2_0} \left\vert x-y \right\vert, \quad \hbox{for every } x, y\in\mathbb{R}^n.
\end{equation}
\end{subequations}
and
\begin{equation} \label{c-CR}
\left\vert c(x)\right\vert\leq\lambda^{-1}\rho_0^{-2}, \quad \hbox{for every } x\in\mathbb{R}^n.
\end{equation}

In what follows we assume $\rho_0=1$.

First of all we consider the case in which
\begin{equation} \label{condition-b}
b\equiv 0
\end{equation}
and we set

\begin{equation} \label{L}
L_0u=\mbox{div}\left(A(x)\nabla_x u\right)+c(x)u,
\end{equation}

Let us denote by $\lambda_j$, with $\lambda_1\leq\cdots \leq\lambda_m\leq 0<\lambda_{m+1}\leq \cdots\leq\lambda_j\leq\cdots$ the eigenvalues associated to the problem

\begin{equation}
\label{1-71CR}
\left\{\begin{array}{ll}
L_0v+\omega q(x)v=0, & \textrm{in }B_2,\\[2mm]
v\in H^1\left(B_2\right),
\end{array}\right.
\end{equation}
and by $e_j(\cdot)$ the corresponding eigenfunctions normalized by

\begin{equation} \label{2-71CR}
\int_{B_2}e^2_j(x)q(x)dx=1.
\end{equation}
In this case the main difference with respect to the case considered above is the presence of non positive eigenvalues $\lambda_1\leq\cdots \leq\lambda_m$. In what follows we indicate the simple changes in the proof of Theorem \ref{5-115} in order to get the same estimate \eqref{SUCP} (with maybe different constants $s_0$ and $C$).
Let $\varepsilon$ and $H$ be the same of  \eqref{4ii-65} and  \eqref{4iii-65}

 Likewise the case $b\equiv 0, c\equiv 0$, the proof can be reduced to the even part $u_+$ with respect to $t$ of solution $u$ of equation \eqref{4i-65CR}. Moreover denoting again by
\begin{equation} \label{3-71CR}
\widetilde{u}(x,t):=\sum_{j=1}^{\infty}\alpha_j e_j(x)\cos\sqrt{\lambda_j} t,
\end{equation}
it is easy to check that instead of Proposition \ref{pag71} we have
\begin{prop}\label{pag71CR}
We have
\begin{equation} \label{1-72CR}
\sum_{j=1}^{\infty}\left(1+|\lambda_j|+\lambda_j^2\right)\alpha^2_j\leq C H^2,
\end{equation}
where $C$ depends on $\lambda, \Lambda$ only.
Moreover, $\widetilde{u}\in \mathcal{W}\left(\mathbb{R};B_2\right)\cap C^0\left(\mathbb{R};H^2\left(B_2\right)\cap H^1_0\left(B_2\right)\right)$ is an even function with respect to variable $t$ and it satisfies
\begin{equation}
\label{2-72CR}
\left\{\begin{array}{ll}
q(x)\partial^2_{t}\widetilde{u}-L_0 \widetilde{u}=0, \quad \hbox{in } B_2\times \mathbb{R},\\[2mm]
\widetilde{u}(\cdot,0)=\widetilde{u}_0, \quad \hbox{in } B_2,\\[2mm]
\partial_t\widetilde{u}(\cdot,0)=0, \quad \hbox{in } B_2.
\end{array}\right.
\end{equation}
 \end{prop}

Similarly to \eqref{unic}, the uniqueness to the Cauchy problem for the equation $q(x)\partial^2_{t}u-L_0u=0$ implies

\[\widetilde{u}(x,t)=u_+(x,t), \quad \mbox{ for } |x|+\lambda^{-1}|t|< 1.\]

Likewise the Section \ref{SUCP_Estimates} we set
\[\widetilde{u}_k:=\widetilde{u}_{\overline{\mu},k},\]
where $\overline{\mu}:=k^{-\frac{1}{6}}$, $k\geq 1$ and $\widetilde{u}_{\mu,k}$ is defined by \eqref{2-76}. In the present case we set, instead of \eqref{3-80},

\begin{gather}
\label{3-80CR}
v_{k}(x,y):=v^{(1)}_{k}(x,y)+v^{(2)}_{k}(x,y),
\end{gather}
where
\begin{subequations}\label{v1-2}
\begin{equation}\label{v1}
v^{(1)}_{k}(x,y)=\sum_{j=1}^{m}\alpha_j \widehat{\varphi}_{\overline{\mu},k}\left(i\sqrt{|\lambda_j|}\right)\cos\left(\sqrt{|\lambda_j|y}\right)e_j(x)\mbox{ , for  } (x,y)\in B_2\times\mathbb{R}
\end{equation}
\begin{equation}\label{v-2}
v^{(2)}_{k}(x,y)=\sum_{j=m+1}^{\infty}\alpha_j \widehat{\varphi}_{\overline{\mu},k}\left(\sqrt{\lambda_j}\right)g_k\left(y\sqrt{\lambda_j}\right)e_j(x)\mbox{ , for  } (x,y)\in B_2\times\mathbb{R}.
\end{equation}
\end{subequations}
and $g_k(z)$ is the same function introduced in Section \ref{SUCP_Estimates}, in particular it satisfies \eqref{2-80}.

Instead of Proposition \ref{2-81prop} we have
\begin{prop}\label{2-81propCR}
Let $v_k$ be defined by \eqref{3-80CR}. We have that $v_{k}(\cdot,y)$ belongs to $H^1\left(B_2\right)$ for every $y\in \mathbb{R}$, $v_{k}(x,y)$ is an even function with respect to $y$ and it satisfies
\begin{equation}
\label{4-5-6-81CR}
\left\{\begin{array}{ll}
q(x)\partial^2_{y}v_{k}+\mbox{div}\left(A(x)\nabla_x v_{k}\right)=f_{k}(x,y), \quad \hbox{in } B_2\times \mathbb{R},\\[2mm]
v_{k}(\cdot,0)=\widetilde{u}_{k},\quad \hbox{in } B_2.
\end{array}\right.
\end{equation}
where
\begin{equation}
\label{1-81CR}
f_{k}(x,y)=\sum_{j=m+1}^{\infty}\lambda_j\alpha_j \widehat{\varphi}_{\overline{\mu},k}\left(\sqrt{\lambda_j}\right)\left(g^{\prime\prime}_k\left(y\sqrt{\lambda_j}\right)-
g_k\left(y\sqrt{\lambda_j}\right)\right)e_j(x).
\end{equation}
Moreover we have

\begin{equation}
\label{3-81CR}
\sum_{j=0}^{2}\|\partial^{j}_yv_{k}(\cdot,y)\|_{H^{2-j}\left(B_2\right)}\leq Ce^{\lambda\sqrt{|\lambda_1|}}H e^{2k} \mbox{, for every  } y\in \mathbb{R},
\end{equation}

\begin{equation}
\label{2-81CR}
\|f_{k}(\cdot,y)\|_{L^2\left(B_2\right)}\leq CH e^{2k}\min\left\{1,\left(4\pi\lambda^{-1}|y|\right)^{2k}\right\} \mbox{, for every  } y\in \mathbb{R},
\end{equation}
where $C$ depends on $\lambda$ and $\Lambda$ only, and

\begin{equation}
\label{1-82CR}
\|v_{k}(\cdot,0)\|_{L^2\left(B_{r_0}\right)}\leq \varepsilon.
\end{equation}
\end{prop}

Instead of Proposition \ref{3-91} we have

\begin{prop}\label{3-91propCR}
Let $v_{k}$ be defined in \eqref{3-80CR}. Then there exists a constant $c$, $0<c<1$, depending on $\lambda$ only such that if $r_0\leq c$, we have
\begin{equation}
\label{1-91}
\|v_{k}\|_{L^2\left(\widetilde{B}_{r_0/4}\right)}\leq Ce^{\lambda\sqrt{|\lambda_1|}}\left(\varepsilon+H\left(C_0r_0\right)^{2k}\right)^{\beta}\left(He^{2k}+H\left(C_0r_0\right)^{2k}\right)^{1-\beta}.
\end{equation}
where $\beta\in (0,1)$, $C$ depend on $\lambda$ and $\Lambda$ only and $C_0=4\pi e \lambda^{-1}$.
\end{prop}

With propositions \ref{pag71CR}, \ref{2-81propCR}, \ref{3-91propCR} at hand and by using Carleman estimate \eqref{6-118}, the proofs of estimates \eqref{SUCP} and \eqref{SUCPBoundary} are straightforward, whenever  \eqref{condition-b} is satisfied.

\bigskip

In the more general case we use a well known trick, see for instance \cite{La-O}, to transform the equation \eqref{4i-65CR} in a self-adjoint equation. Let $z$ be a new variable and denote by $A_{0}(x,z)=\left\{a_{0}^{ij}(x,z)\right\}^{(n+1)}_{i,j=1}$ the real-valued symmetric $(n+1)\times (n+1)$ matrix whose entries are defined as follows. Let $\eta\in C^{1}(\mathbb{R})$ be a function such that $\eta(z)=z$, for $z\in (-1,1)$, and $|\eta(z)|+|\eta'(z)|\leq 2\lambda^{-1}$

\[
a_{0}^{ij}(x,z)=a_{0}^{ij}(x),\quad\hbox{ if }  i,j\in\{1,\ldots,n\},
\]
\[
a_{0}^{(n+1)j}(x,z)= a_{0}^{j(n+1)}(x,z)=\eta(z)b^j(x),\quad\hbox{ if }  1\leq j\leq n,
\]
\[a_{0}^{(n+1)(n+1)}(x,z)=K_0\]
where $K_0=8\lambda^{-3}+1$.
We have that $A_0$ satisfies
\[\lambda_0|\zeta|^2\leq A_0(x,z)\zeta\cdot\zeta\leq \lambda_0^{-1}|\zeta|^2 ,\quad\hbox{ for every }\zeta\in\mathbb{R}^{n+1}\]
and
\[\left\vert A_0(x,z)-A_0(y,w)\right\vert\leq\Lambda_0 \left(\left\vert x-y \right\vert+|z-w|\right), \quad \hbox{for every } (x,z), (y,w)\in\mathbb{R}^{n+1}\]
where $\lambda_0$ depends on $\lambda$ only and $\Lambda_0$ depends on $\lambda, \Lambda$ only.
Denote
\[\mathcal{L}U:=\mbox{div}_{x,z}\left(A_0(x,z)\nabla_{x,z} U\right)+c(x)U\]

It is easy to check that if $u(x,t)$ is a solution of \eqref{4i-65CR} ($\rho_0=1$) then $U(x,z,t):=u(x,t)$ is solution to
\begin{equation*} \label{4i-65CR}
q(x)\partial^2_{t}U-\mathcal{L}U=0, \quad \hbox{in } \widetilde{B}_{1}\times(-\lambda,\lambda).
\end{equation*}

Therefore we are reduced to the case considered previously in this subsection and again the proofs of estimates \eqref{SUCP} and \eqref{SUCPBoundary} are now straightforward.


\begin{thebibliography}{}
\label{bbiibb}

\bibitem[AE]{AE} V. Adolfsson, L. Escauriaza,
$C^{1,\alpha}$ domains and unique continuation at the boundary,
Comm. Pure Appl. Math., 50 (1997), 935-969.

\bibitem[Al]{Al} G. Alessandrini, Examples of instability in inverse boundary-value problems, Inverse
Problems 13, (1997), 887–897.

\bibitem[Al-B-Ro-Ve]{A-B-R-V} G. Alessandrini, E. Beretta, E.
Rosset, S. Vessella, Optimal stability for inverse elliptic boundary
value problems with unknown boundaries, Ann. Scuola Norm. Sup. Pisa
Cl. Sci. 29, (4), (2000), 755-806.

\bibitem[Al-R-Ro-Ve]{A-R-R-V} G. Alessandrini, L. Rondi, E. Rosset,
S. Vessella, The stability for the Cauchy problem for elliptic
equations, Inverse Problems 25 (2009), 1-47.

\bibitem[Al-Ve]{A-V} G. Alessandrini, S. Vessella, Remark on the strong unique continuation property for parabolic operators
Proc. AMS 132, (2004), 499–501

\bibitem[A-K-S]{AKS} N. Aronszajn, A. Krzywicki and J. Szarski, \textit{A unique
continuation theorem for exterior differential forms on riemannian manifolds}%
, Ark. for Matematik, 4, (34), (1962), 417-453.

\bibitem[Ba-Za]{Ba-Za} M. S. Baouendi and E. C. Zachmanoglou, Unique continuation of solutions of partial differential equations and inequalities from manifolds of any dimension. Duke Math. Journal, 45, (1), (1978), 1-13.

\bibitem[Bo]{Bo} J. Boman, A local vanishing theorem for distribution, CRAS Paris,, 315, series I, (1992), 1231-1234.

\bibitem[C-Ro-Ve]{CRoVe1} B. Canuto, E. Rosset, S. Vessella, Quantitative
estimates of unique continuation for parabolic equations and inverse
initial-boundary value problems with unknown boundaries, Trans. Am.
Math. Soc. 354, (2), (2002), 491-535.

\bibitem[CRoVe2]{CRoVe2} B. Canuto, E. Rosset, S. Vessella, A stability
result in the localization of cavities in a thermic conducting
medium, ESAIM: Control Optimization and Calculus of Variations. 7,
2002, 521-565.

\bibitem[Che-Y-Z]{CheYZ} J. Cheng, M. Yamamoto, Q. Zhou, Unique continuation on a hyperplane for wave equation. Chinese Ann. Math. Ser. B 20 (1999), no. 4, 385–392.

\bibitem[Che-D-Y]{CheDY} J. Cheng, G. Ding, M.Yamamoto, Uniqueness along a line for an inverse wave source problem. Comm. Partial Differential Equations 27 (2002), no. 9-10, 2055–2069.

\bibitem[Dc-R]{Dc-R} Di Cristo, M., Rondi, L.: Examples of exponential instability for inverse inclusion and
scattering problems. Inverse Problems 19, 685–701 (2003)
8. Engl, H.W., Langthaler, T., Manselli, P.: On an inverse problem for a nonlinear heat

\bibitem[Dc-R-Ve]{DcRVe} M. Di Cristo, L. Rondi, S. Vessella, Stability
properties of an inverse parabolic problem with unknown boundaries,
Ann. Mat. Pura Appl. (4) 185 (2) (2006) 223-255.

\bibitem[Es-Fe]{EsFe} L. Escauriaza, J. Fernandez, Unique continuation for parabolic operator Ark. Mat. 41, (2003) 35–60

\bibitem[Es-Fe-Ve]{EsFeVe} L. Escauriaza, F. J. Fernandez, S. Vessella, Doubling properties of caloric functions Appl. Anal. vol 85, (2006)
(Special issue dedicated to the memory of Carlo Pucci ed R Magnanini and G Talenti


\bibitem[Es-Ve]{EsVe} L. Escauriaza, S. Vessella, Optimal three cylinder
inequalities for solutions to parabolic equations with Lipschitz leading
coefficients, in G. Alessandrini and G. Uhlmann eds, \textquotedblright
Inverse Problems: Theory and Applications\textquotedblright\ Contemporary
Mathematics 333, (2003), American Mathematical Society, Providence R. I. ,
79-87.


\bibitem[Ev]{Ev} L. C. Evans, Partial differential equations, \textit{American
Mathematical Society, Providence 1998.}

\bibitem[Ga-Li]{GaLi} N. Garofalo, F. H. Lin, Monotonicity properties of variational integrals Ap-weights and unique
continuation Indiana Univ. Math. J. 35, (1986), 245–67


\bibitem[G-T]{GT} D. Gilbarg, N. S. Trudinger, Elliptic partial differential equations of second order, Springer, New York, 1983.

\bibitem[H{\"o}1]{hormanderpaper1} L. H{\"o}rmander, Uniqueness Theorem for Second Order Elliptic Differential
Equations, Comm. Part. Diff. Equations, 8 (1),1983, 21-64.

\bibitem[H{\"o}2]{hormanderpaper2}
L. H{\"o}rmander, On the uniqueness of the Cauchy problem under partial analyticity assumption. Geometrical Optics and related Topics PNLDE (32) Birkh{\"a}user (1997). Editors: Colombini-Lerner.

\bibitem[Is1]{isakovlib2}
V. Isakov, Inverse problems for partial differential equations, volume 12 of Applied Mathematical Sciences, Springer, New York, second edition, 2006.

\bibitem[Is2]{Is2} V. Isakov, On uniqueness of obstacles and boundary conditions from restricted dynamical and scattering data, Inverse Probl. Imaging 2 (2008), no. 1, 151–165.

\bibitem[Jo]{J} F. John, Continuous dependence on data for solutions of partial differential equations
with a prescribed bound, Comm.  Pure and Appl. Math., VOL. XIII, (1960), 551-586


\bibitem[Ko-Ta1]{KoTa1} H. Koch, D. Tataru, Carleman estimates and unique continuation for second-order elliptic equations with nonsmooth coefficients. Comm. Pure Appl. Math. 54 (2001), no. 3, 339–360.

\bibitem[Ko-Ta2]{KoTa2} H. Koch, D. Tataru, Carleman estimates and unique continuation for second order parabolic equations with nonsmooth coefficients. Comm. Partial Differential Equations 34 (2009), no. 4-6, 305–366.

\bibitem[La-O]{La-O}  E. M. Landis, O. A. Oleinik, Generalized analyticity and some related properties of solutions of elliptic
and parabolic equations Russ. Math. Surv. 29, (1974), 195–212.

\bibitem[La]{La}  E. M. Landis, A three sphere theorem Soviet Math. Dokl. 4, (1963), 76–8 (Engl. Transl.)

\bibitem[Le]{Le} G. Lebeau, Un problem d'unicit\'{e} forte pour l'equation des ondes, Comm. Part. Diff. Equat. 24 (1999), 777-783.

\bibitem[Ma]{Ma} K. Masuda, A unique continuation theorem for solutions of wave equation with variable coefficients, J. Math. Anal. Appl. 21, 369-376 (1968)

\bibitem[Ra]{Ra} Rakesh, A remark on unique continuation along and acrosslower dimensional planes for the wave equation. Math. Methods Appl. Sci. 32 (2009),2, 246-252.

\bibitem[Ro-Zu]{Ro-Zu} L. Robbiano and C. Zuily, Uniqueness in the Cauchy problem for operators with with partial holomorphic coefficients. Inventiones Math., 131, 3, 1998, 493-539.

\bibitem[Ta]{Ta} D. Tataru, Unique continuation for solutions to PDE's; between H{\"o}rmander theorem and Holmgren theorem, Comm. Part. Diff. Equat., 20, (1995), 855-884

\bibitem[Ve]{Ve} S. Vessella
Quantitative estimates of unique continuation for parabolic
equations, determination of unknown time-varying boundaries and
optimal stability estimates, Inverse Problems 24, (2008),
pp.~1--81.

\end{thebibliography}
\end{document}